\documentclass[12pt]{amsart}

\usepackage{amsmath,amsfonts,amssymb,amsthm,graphicx,overpic,hyperref}
\usepackage{float,enumitem}
\usepackage{color,xcolor}

\usepackage[margin=1in]{geometry}

\numberwithin{equation}{section}

\newtheorem{thm}{Theorem}[section]
\newtheorem{prop}[thm]{Proposition}

\newtheorem{cor}[thm]{Corollary}
\newtheorem{lem}[thm]{Lemma}

\theoremstyle{definition}

\newtheorem{defn}[thm]{Definition}

\theoremstyle{remark}
\newtheorem{rem}[thm]{Remark}
\newtheorem*{acknowledgements}{Acknowledgements}

\newcommand{\thmref}[1]{Theorem~\ref{#1}}

\newcommand{\propref}[1]{Proposition~\ref{#1}}
\newcommand{\secref}[1]{Section~\ref{#1}}

\newcommand{\lemref}[1]{Lemma~\ref{#1}}
\newcommand{\corref}[1]{Corollary~\ref{#1}}
\newcommand{\figref}[1]{Figure~\ref{#1}}

\newcommand{\eqnref}[1]{Equation~\eqref{#1}}

\newcommand{\nc}{\newcommand}
\nc{\dmo}{\DeclareMathOperator}
\usepackage{dsfont}

\nc{\abs}[1]{\left| #1 \right|}
\nc{\bigO}[1]{O\left(#1\right)}
\nc{\card}[1]{\left|#1\right|}
\nc{\ceil}[1]{\left\lceil #1 \right\rceil}
\nc{\CC}{\mathbb{C}}
\nc{\dilog}{\mathcal{L}}
\nc{\floor}[1]{\left\lfloor #1 \right\rfloor}
\nc{\ind}{\mathds{1}}
\nc{\ZZ}{\mathbb{Z}}
\nc{\len}[1]{\left| #1 \right|}
\nc{\littleo}[1]{o\left(#1\right)}
\dmo{\Mat}{Mat}
\nc{\NN}{\mathbb{N}}
\nc{\norm}[1]{\left|\left| #1 \right|\right|}
\nc{\PP}{\mathbb{P}}
\nc{\QQ}{\mathbb{Q}}
\nc{\RR}{\mathbb{R}}
\renewcommand{\SS}{\mathbb{S}}
\nc{\st}[2]{\left\{\, #1 \,:\, #2\,\right\}}
\dmo{\supp}{supp}
\nc{\tr}[1]{\mathrm{tr}\left(#1\right)}
\nc{\what}{\widehat}
\dmo{\im}{Im}
\dmo{\re}{Re}
\nc{\eps}{\varepsilon}
\dmo{\li}{li}
\dmo{\chr}{chr}

\dmo{\arccosh}{arccosh}
\dmo{\arcsinh}{arcsinh}
\dmo{\area}{area}
\dmo{\conv}{conv}
\dmo{\diam}{diam}
\dmo{\DD}{\mathbb{D}}
\dmo{\dist}{\mathrm{d}}
\nc{\HH}{\mathbb{H}}
\dmo{\Isom}{Isom}
\dmo{\MCG}{MCG}
\dmo{\MPL}{MPL}
\dmo{\Mod}{\mathcal{M}}
\dmo{\PL}{PL}
\nc{\Sphere}{\mathbb{S}}
\dmo{\sys}{sys}
\dmo{\kiss}{kiss}
\dmo{\Teich}{\mathcal{T}}
\nc{\Torus}{\mathbb{T}}
\dmo{\vol}{vol}
\dmo{\WP}{WP}

\dmo{\Aut}{Aut}
\dmo{\Fix}{Fix}
\dmo{\GL}{GL}
\dmo{\id}{Id}
\dmo{\PSL}{PSL}
\dmo{\PGL}{PGL}
\dmo{\Rep}{Rep}
\dmo{\SL}{SL}
\dmo{\SO}{SO}
\dmo{\sym}{\mathfrak{S}}

\dmo{\calA}{\mathcal{A}}
\dmo{\calB}{\mathcal{B}}
\dmo{\calC}{\mathcal{C}}
\dmo{\calD}{\mathcal{D}}
\dmo{\calE}{\mathcal{E}}
\dmo{\calF}{\mathcal{F}}
\dmo{\calG}{\mathcal{G}}
\dmo{\calH}{\mathcal{H}}
\dmo{\calI}{\mathcal{I}}
\dmo{\calJ}{\mathcal{J}}
\dmo{\calK}{\mathcal{K}}
\dmo{\calL}{\mathcal{L}}
\dmo{\calM}{\mathcal{M}}
\dmo{\calN}{\mathcal{N}}
\dmo{\calO}{\mathcal{O}}
\dmo{\calP}{\mathcal{P}}
\dmo{\calQ}{\mathcal{Q}}
\dmo{\calR}{\mathcal{R}}
\dmo{\calS}{\mathcal{S}}
\dmo{\calT}{\mathcal{T}}
\dmo{\calU}{\mathcal{U}}
\dmo{\calV}{\mathcal{V}}
\dmo{\calW}{\mathcal{W}}
\dmo{\calX}{\mathcal{X}}
\dmo{\calY}{\mathcal{Y}}
\dmo{\calZ}{\mathcal{Z}}

\nc{\cech}{\v{C}ech}

\title[The Klein quartic maximizes the multiplicity of $\lambda_1$]{The Klein quartic maximizes the multiplicity of the first positive eigenvalue of the Laplacian}
\author{Maxime Fortier Bourque}
\address{D\'epartement de math\'ematiques et de statistique, Universit\'e de Montr\'eal, 2920, chemin de la Tour, Montr\'eal (QC), H3T 1J4, Canada}
\email{maxime.fortier.bourque@umontreal.ca}

\author{Bram Petri}
\address{Institut de Math\'ematiques de Jussieu--Paris Rive Gauche ; UMR7586, Sorbonne Universit\'e - Campus Pierre et Marie Curie,
4, place Jussieu, 75252 Paris Cedex 05, France}
\email{bram.petri@imj-prg.fr}
\date{\today}

\begin{document}

\begin{abstract}
We prove that Klein quartic maximizes the multiplicity of the first positive eigenvalue of the Laplacian among all closed hyperbolic surfaces of genus $3$, with multiplicity equal to $8$. We also obtain partial results in genus $2$, where we find that the maximum multiplicity is between $3$ and $6$. Along the way, we show that for every $g\geq 2$, there exists some $\delta_g>0$ such that the multiplicity of any eigenvalue of the Laplacian on a closed hyperbolic surface of genus $g$ in the interval $[0,1/4+\delta_g]$ is at most $2g-1$ despite the fact that this interval can contain arbitrarily many eigenvalues. This extends a result of Otal to a larger interval but with a weaker bound, which nevertheless improves upon the general upper bound of S\'evennec. 
\end{abstract}


\maketitle

\section{Introduction}

Given a closed, connected, Riemannian surface $S$, we denote by $\lambda_1(S)$ the smallest positive eigenvalue of the Laplacian $\Delta$ on $S$ and by $m_1(S)$ its multiplicity, that is, the dimension of the corresponding eigenspace. The goal of this paper is to prove a sharp upper bound on $m_1$ for closed hyperbolic surfaces of genus $3$. Our proof combines topological methods, spectral inequalities, ideas from the theory of sphere packings, rigorous computer calculations, and some representation theory.

We start by recalling previous results. A conjecture of Colin de Verdi\`ere (see \cite[p.269]{CdV} and \cite[p.601]{CdV2}) states that
\begin{equation} \label{eq:CdV}
\sup \{ m_1(S) : S \cong \Sigma  \} = \chr(\Sigma)-1
\end{equation}
for every closed connected surface $\Sigma$, where the supremum is over Riemannian surfaces $S$ homeomorphic to $\Sigma$ and $\chr(\Sigma)$ is the chromatic number of $\Sigma$, defined as the largest $n$ such that the complete graph on $n$ vertices embeds in $\Sigma$. By a result of Ringel and Youngs \cite{RY},
\[
\mathrm{chr}(\Sigma) = \left\lfloor \frac12 \left( 7 + \sqrt{49 - 24 \chi(\Sigma)}\right) \right\rfloor
\]
where $\chi$ is the Euler characteristic, except if $\Sigma$ is the Klein bottle, in which case $\mathrm{chr}(\Sigma)=6$. The conjecture is known to hold for the sphere \cite{ChengMultiplicity}, the torus \cite{Besson}, the projective plane \cite{Besson}, and the Klein bottle \cite{CdV2,Nadirashvili}.

In negative Euler characteristic, the best known upper bound is
\begin{equation} \label{eq:sevennec}
m_1(S) \leq 5 - \chi(S),
\end{equation}
due to S\'evennec \cite{Sevennec}, improving upon previous results in \cite{Nadirashvili}, \cite{Besson}, and \cite{ChengMultiplicity}. If $S$ is orientable of genus $g\geq 2$, then inequality \eqref{eq:sevennec} gives $m_1(S) \leq 2g+3$.

As for lower bounds, Colbois and Colin de Verdi\`ere \cite{CCV} constructed a closed hyperbolic surface $S_g$ in every genus $g\geq 3$ such that
\[
m_1(S_g) = \left\lfloor \frac12 \left( 1 + \sqrt{8g + 1}\right) \right\rfloor,
\]
which has the same order of growth as the conjectured upper bound. If we allow other operators than the Laplacian, then this was improved in  \cite[Th\'eor\`eme 1.5]{CdV2}, namely, for any closed surface $\Sigma$, there exists a Riemannian metric on $\Sigma$ and a Schrödinger operator $H$ (Laplacian plus potential) such that the multiplicity of the second smallest eigenvalue of $H$ is equal to $\mathrm{chr}(\Sigma)-1$. Note that S\'evennec's result above holds for any such operator.

For the Laplacian on closed, orientable, connected, hyperbolic surfaces, S\'evennec's bound \eqref{eq:sevennec} was improved by Otal \cite[Proposition 3]{Otal} in a certain range. More precisely, the multiplicity of any eigenvalue of the Laplacian in the interval $(0,1/4]$ is at most $2g-3$. Note that the trivial eigenvalue $0$ is always simple. Shortly after, Otal and Rosas \cite{OtalRosas} proved that the total number of eigenvalues of $\Delta$ in the interval $(0,1/4]$ is at most $2g-3$, counting multiplicity.

Our first result is a slightly weaker version of Otal's bound for a small interval beyond $1/4$. By restricting ourselves to $\lambda_1$, we can also enlarge this interval a little bit at the expense of increasing the upper bound further.

\begin{thm} \label{thm:not_too_large}
For every $g \geq 2$, there exist $0<\delta_g < \eps_g$ such that the following implications hold for every closed hyperbolic surface $S$ genus $g$. If $\lambda$ is any eigenvalue of the Laplacian on $S$ such that $\lambda \leq \frac14+\delta_g$, then its multiplicity is at most $2g-1$. If $\lambda_1(S) < \frac14+\eps_g$, then $m_1(S) \leq 2g$.
\end{thm}

Explicit estimates for $\delta_g$ and $\eps_g$ are given at the end of \secref{sec:sevennec}. Note that the result of Otal and Rosas cannot be extended beyond $1/4$, as a hyperbolic surface of genus $g$ can have as many eigenvalues as we want in the interval $(1/4,1/4+\mu)$ for any $\mu>0$ if it has a sufficiently short closed geodesic (see the proof of Theorem 8.1.2 in \cite[p.219]{Buser}). However, if one assumes a lower bound $c>0$ on the systole of $S$, then the upper bound of $2g-3$ does extend to an interval of the form $(0,1/4+\mu_{g,c}]$ for some explicit $\mu_{g,c}>0$ \cite[Theorem A]{Mondal}. A related improvement says that if $S$ is a closed hyperbolic surface of genus $g$ with systole at least $2.317$, then $S$ has strictly less than $2g-3$ Laplacian eigenvalues in $(0,1/4]$ \cite[Section 10.3]{FBP} (see also \cite{Huber} and \cite{Jammes} for earlier versions of this bound with the constants $\sqrt{8}$ and $3.46$ respectively). 

We then focus on the case $g=3$. Among all closed hyperbolic surfaces of genus $3$, the Klein quartic $K$ has been conjectured to maximize $\lambda_1$ in \cite[Conjecture 5.2]{Cook}. Numerical calculations from \cite[Table D.1]{Cook} (reproduced in \cite[Section 4.3]{KMP} with higher precision) suggest that
\[
\lambda_1(K) \approx 2.6779 \quad \text{and} \quad m_1(K)=8=2g+2.
\]
In particular, $K$ does not satisfy the conclusion of \thmref{thm:not_too_large} (because it falls outside of its range). Our second and main result patches this gap by reducing S\'evennec's upper bound of $2g+3=9$ down to $8$ for all closed hyperbolic surfaces of genus $3$. We also confirm that $m_1(K)=8$ following Cook's earlier estimate $6 \leq m_1(K) \leq 8$ \cite[Corollary 4.5]{Cook}, though we do not know if $K$ is the only surface attaining this value.

\begin{thm} \label{thm:main}
If $S$ is a closed hyperbolic surface of genus $3$, then $m_1(S)\leq 8$ and equality holds if $S$ is the Klein quartic.
\end{thm}

Observe that for a surface of genus $3$ we have $\chr(S)-1 =8$. Thus, according to Colin de Verdi\`ere's conjecture, \thmref{thm:main} should hold for all Riemannian metrics, not just hyperbolic ones. However, our techniques are very specific to metrics of constant curvature.

Among other things, the proof of \thmref{thm:main} uses the recent upper bound
\[
\lambda_1(S) \leq \frac{16 (4 - \sqrt{7})\pi}{\area(S)}
\]
for closed, oriented, Riemannian surfaces of genus $3$ \cite{Ros}, which translates to
\begin{equation} \label{eq:ros}
\lambda_1(S) \leq 2(4 - \sqrt{7}) \approx 2.708497...
\end{equation}
when restricted to hyperbolic surfaces of genus $3$ by the Gauss--Bonnet formula.

In order to bound $m_1(S)$ when $\lambda_1(S)$ is between $\frac14+\eps_3$ and $2(4 - \sqrt{7})$, we adapt a strategy used by Cohn and Elkies to bound the density of sphere packings in Euclidean spaces \cite{CohnElkies}. Namely, we use linear programming in order to extract the best possible bounds on $m_1$ from the Selberg trace formula. The bounds we obtain depend on $\lambda_1$ and blow up as $\lambda_1$ approaches $1/4$ or $+\infty$. This is why \thmref{thm:not_too_large} and inequality \eqref{eq:ros} are required.

Using the same methods, we also improve S\'evennec's bound by $1$ in genus $2$.

\begin{thm} \label{thm:genus2}
If $S$ is a closed hyperbolic surface of genus $2$, then $m_1(S) \leq 6=2g+2$.
\end{thm}

In a private communication, A. Strohmaier has told us that he conjectures that the Bolza surface $B$ uniquely maximizes $m_1$ among closed hyperbolic surfaces of genus $2$, based on numerical calculations carried out with the computer program \texttt{Hypermodes} that he developed with V. Uski. That $B$ maximizes $\lambda_1$ among hyperbolic surfaces was conjectured in \cite{StrohmaierUski}.

In \cite{Jenni}, Jenni claimed that $m_1(B)=3$, but Cook \cite{Cook} found a mistake in the proof. In trying to fix Jenni's argument, Cook showed that $3 \leq m_1(B) \leq 4$. We use our techniques to improve Cook's upper bound and confirm that $m_1(B)=3$ in \secref{sec:bolza}.

Since a surface of genus $2$ satisfies $\chr(\Sigma)-1=7$, \thmref{thm:genus2} implies that equation \eqref{eq:CdV} fails when restricted to hyperbolic metrics. For Riemannian metrics, the largest value of $m_1$ achieved thus far in genus $2$ is $5$ \cite[Proposition 4]{NayataniShoda}. As we mentioned earlier, for Schr\"odinger operators the multiplicity of the second smallest eigenvalue can reach $\chr(\Sigma)-1$, which coincides with S\'evennec's upper bound of $2g+3$ when $g=2$.

\begin{acknowledgements}
We thank Joseph Cook, Dmitry Jakobson, Iosif Polterovich, and Alexander Strohmaier for useful conversations and references. We also thank Chul-hee Lee for catching errors in the previous version of the proofs of Propositions \ref{prop:less6} and \ref{prop:less12}, and the files \texttt{verify2.sage} and \texttt{verify3.sage} and Benno Wendland for catching an error in Section \ref{sec:one_dim_reps_Klein}.
\end{acknowledgements}

\section{Upper bound for smallish eigenvalues}
\label{sec:sevennec}

The Laplacian of a $C^2$ real-valued function $f$ on a connected Riemannian surface $S$ is defined as $\Delta f = - \mathrm{div}(\mathrm{grad} f)$. If $S$ is not closed but is homeomorphic to the interior of a compact surface, we say that $f$ satisfies the \emph{Dirichlet boundary condition} if $f(x) \to 0$ as $x$ approaches $\partial S$. 

Given an eigenfunction $f$ of the Laplacian, the zero level set \[\calZ(f):=\{ x \in S : f(x) = 0 \}\] is called its \emph{nodal set}. A theorem of Cheng  \cite{ChengMultiplicity} states that if $S$ is closed or if $f$ satisfies the Dirichlet boundary condition, then $\calZ(f)$ is a locally finite graph where each vertex has positive even degree and the edges adjacent to it meet at equal angles. If $S$ is closed, then this graph is equal to a union of finitely many smooth closed curves called \emph{nodal lines}.

The connected components of the complement of the nodal set are called \emph{nodal domains}. As usual, we list the eigenvalues of $\Delta$ on $S$ (with Dirichlet boundary conditions if $S$ is not closed) as 
\[
0 \leq \lambda_0(S) < \lambda_1(S) \leq \ldots
\]
with each eigenvalue repeated according to its multiplicity. Courant's nodal domain theorem  \cite[p.20]{Chavel} says that any eigenfunction associated to $\lambda_j(S)$ has at most $j+1$ nodal domains. Moreover, $\lambda_0(S)$ is the only eigenvalue that admits eigenfunctions of constant sign. Thus, for any eigenfunction $f$ associated to an eigenvalue $\lambda > \lambda_0(S)$, the open sets
\[
U^+(f) = \{ x \in S : f(x) > 0 \} \quad \text{and} \quad U^-(f)=\{ x \in S :  f(x) <0 \}
\]
are both nonempty. If $\lambda = \lambda_1(S)$, then there are exactly two nodal domains, so $U^+(f)$ and $U^-(f)$ are both connected.

Assume now that $S$ is closed, let $\calE$ be the eigenspace of the Laplacian on $S$ associated to an eigenvalue $\lambda>0=\lambda_0(S)$, and let 
\[
\SS(\calE)= (\calE\setminus\{0\}) / \RR_{>0} \quad\text{and}\quad  \PP(\calE)= (\calE\setminus\{0\}) / (\RR\setminus\{0\}).
\] 
That is, $\SS(\calE)$ is homeomorphic to the unit sphere in $\calE$ with respect to any norm and $\PP(\calE)$ is its projectivization. The double covering $\SS(\calE) \to \PP(\calE)$ identifies the eigenfunctions $f$ and $-f$, which have the same sets $U^+$ and $U^-$ but swapped. It is from this point of view that S\'evennec studies these spaces in \cite{Sevennec}. 

The double covering $\SS(\calE) \to \PP(\calE)$ can also be described by a \cech{} cohomology class $\alpha \in \check{H}^1(\PP(\calE))$, where coefficients are taken in $\ZZ/2\ZZ$. It is a standard fact that \[\check{H}^*(\PP(\calE)) = (\ZZ/2\ZZ)[\alpha] / \langle \alpha^m \rangle\] where $m$ is the dimension of $\calE$, which is the same as the multiplicity of $\lambda$. In particular, if one can show that $\alpha^N$ is equal to zero in $\check{H}^N(\PP(\calE))$, then this gives the upper bound $m \leq N$. In \cite{Sevennec}, S\'evennec introduces the following slightly weaker notion.

\begin{defn}
Let $X$ be a topological space. A cohomology class $u \in \check{H}^N(X)$ is \emph{weakly zero} if $F^*u = 0$ in $\check{H}^N(M)$ for every continuous map $F:M \to X$ from a finite dimensional metric space $M$.
\end{defn}

It is still true that if $\alpha^N$ is weakly zero, then $m \leq N$, since $\PP(\calE)$ is a finite dimensional metric space. S\'evennec's strategy for finding such an $N$ is to partition $\PP(\calE)$ according to the topology of the nodal domains and to analyze each piece separately. These bounds can then be combined together using the following lemma \cite[Lemma 8]{Sevennec}.

\begin{lem}[S\'evennec]  \label{lem:cup_product}
Let $Y$ be a topological space and let $Y_i\subset Y$ be arbitrary subsets such that $Y =\bigcup_{i=0}^k Y_i$. If $u_0, \ldots, u_k \in \check{H}^*(Y)$ are such that $u_i|_{Y_i}$ is weakly zero in $\check{H}^*(Y_i)$ for every $i$, then the cup product $u_0 \cdots u_k$ is weakly zero in $\check{H}^*(Y)$.
\end{lem}

For a topological space $U$, let $b_1(U)$ be its first Betti number. The partition that S\'evennec considers is defined as follows. For each integer $i \geq 0$, let
\[
\widetilde{Y}_i = \{ [f] \in \SS(\calE) : b_1(U^+(f))+b_1(U^-(f)) = i \}
\]
and let $Y_i$ be its projection under the double covering $\SS(\calE) \to \PP(\calE)$. S\'evennec shows in \cite[Lemma 12]{Sevennec} that if $\calE$ is the eigenspace corresponding to $\lambda = \lambda_1(S)$, then $\widetilde{Y}_i$ is empty if $i>b_1(S)$, so that $\PP(\calE) = \bigcup_{i=0}^{b_1(S)} Y_i$.

By applying \lemref{lem:cup_product} with each $u_i$ equal to the restriction of $\alpha^{N_i}$ to $Y_i$ for some power $N_i$, one obtains that $\alpha^{\sum N_i}$ is weakly zero in $\check{H}^*(\PP(\calE))$ and hence that $m_1(S) \leq \sum_{i=0}^{b_1(S)} N_i$. This reduces the problem to finding the smallest $N_i$ such that $\alpha^{N_i}|_{Y_i}$ is weakly zero. If $S$ is orientable of genus $g \geq 2$, then S\'evennec shows that $N_0 \leq 3$ and $N_i = 1$ for every $i>0$ \cite[Theorem 9]{Sevennec}.  Note that S\'evennec actually proves that $\alpha|_{Y_i}=0$ if $i>0$; only the case $i=0$ requires the notion of weak vanishing. Since $b_1(S) = 2g$, the resulting bound is $m_1(S) \leq \sum_{i=0}^{2g} N_i \leq  2g+3$.

To improve on this for small eigenvalues on hyperbolic surfaces, Otal \cite{Otal} used a different partition of $\PP(\calE)$ defined as follows. For each integer $i$, let
\[
\widetilde{X}_i = \{ [f] \in \SS(\calE) : \chi(U^+(f))+\chi(U^-(f)) = i \}
\]
where $\chi$ is the Euler characteristic and let $X_i$ be the image of $\widetilde{X}_i$ in $\PP(\calE)$. Since
\[
2-2g = \chi(S) = \chi(U^+(f)) + \chi(U^+(f)) + \chi(\calZ(f))  \quad \text{and} \quad \chi(\calZ(f)) \leq 0,
\]
we have that $X_i$ is empty if $i < 2-2g$. Otal \cite[Lemma 1]{Otal} proved that $\chi(U^\pm(f)) < 0$ if $\lambda \in (0, 1/4]$ so that $X_i$ is empty whenever $i>-2$ in this case. He also showed that the covering $\widetilde{X}_i \to X_i$ is trivial, or equivalently that $\alpha |_{X_i} = 0$ for every $i$ between $2-2g$ and $-2$. By \lemref{lem:cup_product}, it follows that $\alpha^{2g-3}$ is weakly zero and hence that the multiplicity of $\lambda$ is at most $2g-3$.

The value $1/4$ comes into play because it is the bottom of the spectrum of the Laplacian on the hyperbolic plane and on any complete hyperbolic cylinder. The monotonicity of eigenvalues with respect to inclusion implies that no nodal domain of an eigenfunction $f$ associated to an eigenvalue in $[0,1/4]$ can be homeomorphic to a disk or an annulus. If we increase the cutoff a little bit beyond $1/4$, then we can still rule out disks, but not annuli.

\begin{lem} \label{lem:no_disk}
Let $g \geq 2$ and let $D$ and $E$ be hyperbolic disks of area $4 \pi (g-1)$ and $2 \pi (g-1)$ respectively. Suppose that $S$ is a closed hyperbolic surface of genus $g$ and that $f$ is an eigenfunction of the Laplacian on $S$ corresponding to an eigenvalue $\lambda$.  If $\lambda \leq \lambda_0(D)$ where $\lambda_0$ denotes the smallest Dirichlet eigenvalue, then no nodal domain of $f$ can be contracted to a point in $S$. If $\lambda < \lambda_0(E)$, then not all nodal domains of $f$ can contracted to points in $S$.
\end{lem}
\begin{proof}
We proceed by contraposition for each implication, starting with the first. Suppose that $U$ is a nodal domain of $f$ that can be contracted to a point. Since the restriction of $f$ to $U$ is an eigenfunction for the Laplacian on $U$ which vanishes along the boundary, we have
\[
\lambda \geq \lambda_0(U).
\]
In fact, equality holds by Courant's nodal domain theorem. 

Since $U$ lifts to the hyperbolic plane, which satisfies an isoperimetric inequality, the Faber--Krahn inequality \cite[p.87]{Chavel} tells us that 
\[\lambda_0(U) \geq \lambda_0(D')\]
where $D'$ is a round hyperbolic disk with the same area as $U$. Since $S$ is not contractible, $f$ must have more than one nodal domain so that $\area(U) < \area(S) = 4\pi(g-1)$. If $D$ is a disk of area $4\pi(g-1)$, then we have $\lambda_0(D') > \lambda_0(D)$ by the strict monotonicity of eigenvalues \cite[p.18]{Chavel}, so that $\lambda > \lambda_0(D)$.

For the second implication, suppose that all nodal domains of $f$ can be contracted to points in $S$. Once again, there are at least two nodal domains. Let $U$ be a nodal domain of $f$ of smallest area. Then $\area(U) \leq \area(S)/2 =2\pi(g-1)$. The same reasoning as above can be applied with a disk $E$ of area $2\pi(g-1)$ to yield $\lambda = \lambda_0(U) \geq \lambda_0(E)$.
\end{proof}

 As eigenvalues in the interval $[0,1/4]$ are called \emph{small}, we think of those below $\lambda_0(E)$ as \emph{smallish}, hence the title of this section.

\begin{rem}
For an essential annulus $U$ in a hyperbolic surface $S$, there is a lower bound on $\lambda_0(U)$ which is larger than $1/4$, but by an amount that depends on the length of the closed geodesic homotopic to $U$ \cite[Proposition 1.4]{Mondal}.
\end{rem}

In the first case of \lemref{lem:no_disk}, we can conclude that each nodal domain of $f$ is \emph{incompressible}, meaning that its fundamental group injects into the fundamental group of $S$ under the inclusion map.

\begin{cor} \label{cor:incompressible}
Let $S$ be a closed hyperbolic surface of genus $g\geq 2$ and suppose that $f$ is an eigenfunction of the Laplacian on $S$ corresponding to an eigenvalue $\lambda \leq \lambda_0(D)$ where $D$ is a hyperbolic disk of area $4\pi(g-1)$. Then each nodal domain of $f$ is incompressible in $S$.
\end{cor}

\begin{proof}
Suppose that $U$ is a nodal domain of $f$ that is not incompressible. Then there is a closed curve $\gamma$ in $U$ which is homotopic to a point in $S$ but not in $U$. This implies that $\gamma$ bounds a union of topological disks in $S$. Some component of the complement of $U$ must be contained in one of these disks, for otherwise $\gamma$ would be contractible in $U$. Since every component of the complement of $U$ is the closure of a union of nodal domains, we conclude that $f$ has a nodal domain contained in a disk, which contradicts \lemref{lem:no_disk}.
\end{proof}

For the statement of \thmref{thm:not_too_large}, we set $\delta_g := \lambda_0(D) - \frac14$ and $\eps_g:=\lambda_0(E) - \frac14$ where $D$ and $E$ are hyperbolic disks of area $4\pi(g-1)$ and $2\pi(g-1)$ respectively, as in \lemref{lem:no_disk}. \thmref{thm:not_too_large} then follows from the above observations and the work of S\'evennec and Otal.

\begin{proof}[Proof of \thmref{thm:not_too_large}]
Suppose that $\lambda \leq \frac14+ \delta_g = \lambda_0(D)$, let $\calE$ be the corresponding eigenspace, and let $f \in \calE$. By \lemref{lem:no_disk}, no nodal domain of $f$ is homeomorphic to a disk. It follows that
\[
\chi(U^+(f))+\chi(U^-(f)) \leq 0
\]
and hence that $\PP(\calE) = \bigcup_{i={2-2g}}^{0} X_i$. By \corref{cor:incompressible}, all the components of $U^+(f)$ and $U^-(f)$ are incompressible. For $i<0$, the same argument as on pages 692-693 of \cite{Otal} applies verbatim to show that the covering $\widetilde{X}_i \to X_i$ is trivial. If $i=0$, then $U^+(f)$ and $U^-(f)$ are both unions of annuli, which could be isotopic to each other. However, in that case the pair $(U^+(f),U^-(f))$ is not isotopic to $(U^-(f),U^+(f))$ \cite[Proposition 15]{Sevennec}, so Otal's argument (or \cite[Corollary 17]{Sevennec}) still implies that the covering $\widetilde{X}_i \to X_i$ is trivial. We conclude that the multiplicity of $\lambda$ is at most $2g-1$.

If instead we assume that $\lambda = \lambda_1(S) < \frac14+ \eps_g=\lambda_0(E)$, then \lemref{lem:no_disk} still tells us that at most one of $U^+(f)$ or $U^-(f)$ is homeomorphic to a disk. It follows that \[b_1(U^+(f))+b_1(U^-(f)) \geq 1\] so that $\PP(\calE) = \bigcup_{i=1}^{2g} Y_i$. By appealing to S\'evennec's result that $\alpha|_{Y_i} = 0$ for $i>0$ \cite[Theorem 9]{Sevennec}, we deduce that $\alpha^{2g}$ is weakly zero, and hence that $m_1(S) \leq 2g$.
\end{proof}

Since the area of a hyperbolic disk $D_R$ of radius $R>0$ is equal to $4\pi \sinh^2(R/2)$, the radius of a hyperbolic disk of area $A$ is equal to $2\arcsinh(\sqrt{A/(4\pi)})$. When $R$ is large, the best lower bound that we are aware of on the smallest Dirichlet eigenvalue of $D_R$ is
\[
\lambda_0(D_R) \geq \frac14 + \frac{\pi^2}{R^2} - \frac{4\pi^2}{R^3}
\]
due to Savo \cite[Theorem 5.6]{Savo} while when $R$ is small, the inequality
\begin{equation} \label{eq:arta}
\lambda_0(D_R) \geq \frac14 + \left( \frac{\pi}{2R}\right)^2
\end{equation}
of Artamoshin \cite[Theorem 4.0.2(a)]{Arta} is better.  In the other direction, the best known upper bound is
\begin{equation} \label{eq:gage}
\lambda_0(D_R) \leq \frac14 + \frac{\pi^2}{R^2} - \frac{1}{4\sinh^2(R)},
\end{equation}
due to Gage \cite[Theorem 5.2]{Gage}.

In this paper, we will only apply \thmref{thm:not_too_large} in genus $2$ and $3$, where Artamoshin's inequality yields
\[
\delta_g \geq \left(\frac{\pi}{4\arcsinh(\sqrt{g-1})} \right)^2
\quad \text{and} \quad
\eps_g \geq \left(\frac{\pi}{4\arcsinh\left(\sqrt{(g-1)/2}\right)} \right)^2.
\]
Savo's inequality yields a better estimate on $\delta_g$ when $g\geq 53$ and on $\eps_g$ when $g\geq 104$. 

It is interesting to compare $\frac14+\eps_g=\lambda_0(E)$ with upper bounds on $\lambda_1(S)$. A result of Cheng \cite[Theorem 2.1]{ChengComparison} states that
\[
\lambda_1(S) \leq \lambda_0(D_{\diam(S)/2})
\]
where $\diam(S)$ is the diameter of $S$. Since a closed disk of radius $\diam(S)$ covers all of $S$, we have
\[
4\pi \sinh^2(\diam(S)/2) = \area(D_{\diam(S)}) \geq \area(S) = 4\pi(g-1)
\]
and hence $\diam(S)/2 \geq \arcsinh(\sqrt{g-1})$.
Combining this with Gage's inequality \eqref{eq:gage} yields
\begin{equation} \label{eq:cheng_gage}
\lambda_1(S) - \frac14 \leq \frac{\pi^2}{\arcsinh(\sqrt{g-1})^2} - \frac{1}{(g-1)},
\end{equation}
which is asymptotically $4$ times larger than Savo's lower bound for $\eps_g=\lambda_0(E)-\frac14$ as $g\to\infty$. See \cite[Theorem 8.3]{FBP} for an improvement of the upper bound on $\lambda_1(S)- \frac14$ by a factor of $4$, leading to the improved bound $m_1(S)\leq 2g-1$ when $g$ is sufficiently large.

\section{Linear programming bounds in genus three}

The goal of this section is to prove the first half of \thmref{thm:main}, which we restate here.

\begin{thm} \label{thm:at_most_8}
If $S$ is a closed hyperbolic surface of genus $3$, then $m_1(S)\leq 8$.
\end{thm}

To prove this, we may assume that $\lambda_1(S)$ is between 
\[
\frac14+ \eps_3 \geq 1.044071\ldots \quad \text{and} \quad 2(4-\sqrt{7}) \approx 2.708497\ldots
\]
Indeed, \thmref{thm:not_too_large} implies that $m_1(S) \leq 6$ whenever $\lambda_1(S) < \frac14+ \eps_3$ and  inequality \eqref{eq:ros} states that $\lambda_1(S) \leq 2(4-\sqrt{7})$ for every closed hyperbolic surface $S$ of genus $3$. To be safe, we enlarge this interval to $[1.04,2.71]$. Our proof that $m_1(S) \leq 8$ whenever $\lambda_1(S) \in [1.04,2.71]$ is based on the Selberg trace formula, which we now discuss.

\subsection{The Selberg trace formula} \label{subsec:STF}

Given an integrable function $f: \RR \to \CC$, we define its Fourier transform by
\[
\calF[f](\xi):=\what{f}(\xi) := \frac{1}{\sqrt{2\pi}}\int_{-\infty}^\infty f(x)e^{-i \xi x} dx.
\]
for every $\xi \in \CC$ for which the integral exists. This is not the only convention in use for the Fourier transform, but one of its advantages is that $\calF\left[\calF[f]\right](x) = f(-x)$ almost everwhere whenever $\what{f}$ is integrable on $\RR$. In particular, the eigenvalues of $\calF$ are $\{\pm 1,\pm i\}$. In this paper, we will use real-valued functions $f$ that are linear combinations of $\pm 1$ eigenfunctions of $\calF$, rendering the calculation of $\what{f}$ immediate.

We say that a function $f:\RR \to \CC$ is \emph{admissible} if it is continuous, even, integrable, and $\what{f}$ is defined and holomorphic in the strip 
\[\left\{ \zeta \in \CC : |\im \zeta | < \frac12 + \eps \right\}\] for some $\eps>0$ and satisfies the decay condition $\what{f}(\zeta) = O((1 + |\zeta|^2)^{-(1+\eps)})$ there.

Given a closed hyperbolic surface $S$, we denote by $\sigma(S)$ the multiset of eigenvalues of the Laplacian on $S$ repeated according to their multiplicity. We also define the function $r:[0,\infty) \to \RR \cup i \left[0,\frac{1}{2}\right]$ (where $i = \sqrt{-1}$) by
\[
r(\lambda) = \begin{cases}
i \sqrt{\frac{1}{4}-\lambda} & \text{if } 0 \leq \lambda < \frac{1}{4} \\
\sqrt{\lambda - \frac{1}{4}} & \text{if } \lambda \geq \frac{1}{4}.
\end{cases}
\]

The set of oriented closed geodesics in $S$ is denoted $\calC(S)$. The length of a closed geodesic $\gamma$ is $\ell(\gamma)$ and its \emph{primitive length} is $\Lambda(\gamma) = \ell(\alpha)$ where $\alpha$ is the unique closed geodesic such that $\gamma = \alpha^k$ with $k \geq 1$ maximal.

With the above conventions, the Selberg trace formula states that
\begin{equation} \label{eq:STF}
\sum_{\lambda \in \sigma(S)} \what{f}(r(\lambda)) =  2(g-1)\int_0^\infty r \what{f}(r) \tanh(\pi r) dr+ \frac{1}{\sqrt{2\pi}}\sum_{\gamma \in \calC(S)} \frac{\Lambda(\gamma) f(\ell(\gamma))}{2 \sinh(\ell(\gamma)/2)} 
\end{equation}
for every closed hyperbolic surface $S$ of genus $g$ and every admissible function $f$ \cite[p.253]{Buser}. Note that Buser's convention for the Fourier transform in \cite{Buser} is different from ours, so the formula we obtain is slightly different. If we take take $f$ to be the same as what Buser calls $g$, then the function $h$ in Buser's notation is equal to $\sqrt{2 \pi}\what{f}$. \eqnref{eq:STF} is obtained by dividing Buser's formula throughout by $\sqrt{2 \pi}$ (note that $\area(S) = 4\pi (g-1)$ and $\int_{-\infty}^\infty r \what{f}(r) \tanh(\pi r) dr = 2 \int_0^\infty r \what{f}(r) \tanh(\pi r) dr$).

It will be convenient to use shorthand notation for the different terms in the Selberg trace formula. We will write
\begin{align*} 
\calS &= \sum_{\lambda \in \sigma(S)} \what{f}(r(\lambda))\\
\calI &= 2(g-1)\int_0^\infty r \what{f}(r) \tanh(\pi r) dr \\
\calG &= \frac{1}{\sqrt{2\pi}}\sum_{\gamma \in \calC(S)} \frac{\Lambda(\gamma)}{2 \sinh(\ell(\gamma)/2)} f(\ell(\gamma))
\end{align*}
for the \emph{spectral}, \emph{integral}, and \emph{geometric} terms respectively, so that the Selberg trace formula becomes $\calS = \calI + \calG$.

\subsection{The Cohn--Elkies strategy}

If we assume something about the sign of $f$ and $\what{f}\circ r$ on the length spectrum and the eigenvalue spectrum of $S$, then we can obtain various inequalities involving these spectra. The first versions of this strategy, which goes under the name of \emph{linear programming}, were introduced in \cite{Delsarte} to obtain bounds on error-correcting codes and spherical codes. Our version for hyperbolic surfaces is more closely related to a paper of Cohn and Elkies \cite{CohnElkies} where the Poisson summation formula was used to bound the density of Euclidean sphere packings. 

To bound the multiplicity $m_1(S)$ of $\lambda_1(S)$ for a closed hyperbolic surface $S$, we will use the strategy suggested by the following lemma. Here $\sys(S) = \min \{ \ell(\gamma) : \gamma \in \calC(S) \}$ is the minimal length of a closed geosedic in $S$.

\begin{lem} \label{lem:upper}
Let $S$ be a closed hyperbolic surface of genus $g \geq 2$. Suppose that $0 < a \leq \lambda_1(S)$ and $a<b$, and let $f$ be an admissible function such that $\what{f}(r(\lambda)) \geq c > 0$ for every $\lambda \in [a,b]$, $\what{f}(r(\lambda)) \geq 0$ for every $\lambda \geq b$, and $f(x)\leq 0$ for every $x \geq \sys(S)$. Then the number of eigenvalues of the Laplacian $\Delta$ on $S$ in the interval $[a,b]$ counting multiplicity is at most
\begin{equation} \label{eq:lem_upper}
\frac{1}{c}\left( 2(g-1) \int_0^\infty r \what{f}(r) \tanh(\pi r)\,dr + \frac{1}{\sqrt{2\pi}}\sum_{\gamma \in P} \frac{\ell(\gamma) f(\ell(\gamma))}{\sinh(\ell(\gamma)/2)} - \what{f}(i/2) \right)
\end{equation}
for any set $P$ of unoriented primitive closed geodesics in $S$. 
\end{lem}
\begin{proof}
By hypothesis, the only eigenvalue of in the interval $[0,a)$ is the trivial eigenvalue $\lambda_0(S)=0$ and $\what{f}(r(\lambda)) \geq 0$ for all other eigenvalues. We thus have
\[
c \cdot \left|\sigma(S) \cap[a,b] \right|  \leq \sum_{\lambda \in \sigma(S) \cap[a,b]} \what{f}(r(\lambda))
\leq \calS  - \what{f}(r(0)) = \calI + \calG - \what{f}(i/2).
\]

Since $f$ is non-positive on the length spectrum of $S$, the last quantity does not decrease if we omit terms in the geometric sum $\calG$. The advantage of restricting to primitive curves is that we have $\Lambda(\gamma)=\ell(\gamma)$ for these. By taking the two possible orientations along each curve in $P$, we get
\[
\calG = \frac{1}{\sqrt{2\pi}}\sum_{\gamma \in \calC(S)} \frac{\Lambda(\gamma)f(\ell(\gamma))}{2\sinh(\ell(\gamma)/2)}  \leq \frac{1}{\sqrt{2\pi}}\sum_{\gamma \in P} \frac{\ell(\gamma)f(\ell(\gamma))}{\sinh(\ell(\gamma)/2)} .
\]
The desired inequality is obtained upon dividing by $c$.
\end{proof}

To prove \thmref{thm:at_most_8}, it suffices to apply \lemref{lem:upper} to each subinterval in a partition of $[1.04,2.71]$ such that the resulting upper bound is strictly smaller than $9$ for each subinterval. The subintervals we use are $[1.04,1.857]$ and $[1.857,2.71]$. Since we cannot assume anything about the length spectrum of $S$ for this, we will use admissible functions $f$ that are non-positive everywhere in $\RR$ and will take $P$ to be the empty set, in which case the upper bound from \lemref{lem:upper} simplifies to
\begin{equation} \label{eq:upper_simplified}
m_1(S) \leq (\calI - \what{f}(i/2))/c
\end{equation}
whenever $\lambda_1(S) \in [a,b]$. The more general version stated above will be used in \secref{sec:klein} to determine the multiplicity of $\lambda_1$ on the Klein quartic.

The idea of linear programming is to optimize over the space $\calW$ of functions $f$ that satisfy the given conditions in order to extract the best possible upper bound. It is called linear programming because the space $\calW$ is a convex cone (if we think of \[c = \min \{ \what{f}(r(\lambda)) : \lambda \in [a,b] \} \] as a function of $f$), except that it is infinite-dimensional and is described by infinitely many inequalities. As such, linear optimization techniques do not apply directly, so we adapt the approach used by Cohn and Elkies in \cite{CohnElkies} instead. That is, we restrict to functions $f$ of a specific form involving polynomials.

The \emph{Hermite polynomials} can be defined by the formula
\[
H_n(x) = (-1)^n e^{x^2} \frac{d^n e^{-x^2}}{dx^n}
\]
for every $x \in \RR$ and $n \geq 0$. The corresponding \emph{Hermite functions} $h_n(x) = H_n(x) e^{-x^2/2}$ enjoy the property that $\what{h_n}(\xi) = i^{n}h_n(-\xi)$ for every $\xi \in \CC$ \cite[Equation (4.12.3)]{Lebedev}. We will restrict our attention to the even functions $h_{2n}$ which satisfy $\what{h_{2n}}(\xi) = (-1)^n h_{2n}(\xi)$. Being even polynomials, the $H_{2n}$ can be expressed as polynomials in $x^2$. Indeed, we have
\[
H_{2n}(x) = (-1)^n 4^{n} n! L_n^{-1/2}(x^2)
\]
where the $L_n^{-1/2}$ are Laguerre polynomials \cite[Equation (4.19.5)]{Lebedev}.

The functions we will use in \lemref{lem:upper} are finite linear combinations of the form 
\[
f(x) = \sum_{n=0}^N c_n h_{2n}(x) = \left(\sum_{n=0}^N s_n L_n^{-1/2}(x^2) \right) e^{-x^2/2} =: u(x^2) e^{-x^2/2} 
\]
for some $s_n \in \QQ$, where $c_n = s_n / ((-1)^n 4^{n} n!)$, so that
\begin{align*}
\what{f}(\xi)& =\sum_{n=0}^N (-1)^n c_n h_{2n}(\xi) \\
&= \left(\sum_{n=0}^N (-1)^n s_n L_n^{-1/2}(\xi^2) \right) e^{-\xi^2/2} \\
&=: v(\xi^2) e^{-\xi^2/2}
\end{align*} 
for some polynomials $u$ and $v$. Note that $\what{f}$ is holomorphic in the entire complex plane and decays super-exponentially fast as $|\re \zeta| \to \infty$, so that $f$ is automatically admissible. The tricky part is to make sure that $f$ and $\what{f}$ satisfy the sign conditions required by \lemref{lem:upper}. In order to prevent sign changes, we choose the coefficients $s_n$ in such a way that $f$ and $\what{f}$ have double zeros at prescribed places and we verify that they do not have other zeros in the relevant intervals. We then minimize the right-hand side of inequality \eqref{eq:upper_simplified} by varying the set of prescribed double zeros. The double zeros we prescribe are rational and the coefficients $s_n$ are solved for exactly. This is the other reason behind our convention for the Fourier transform: the Laguerre polynomials $L_n^{-1/2}$ have rational coefficients. We can therefore use Sturm's theorem \cite{Sturm} to count the number of zeros of the polynomials $u$ and $v$ over the desired intervals. Once we know that $f$ and $\what{f}$ satisfy the hypotheses of \lemref{lem:upper}, then all that is left to do is to get a rigorous upper bound on the right-hand side of \eqref{eq:upper_simplified}, which we do using interval arithmetic.

\subsection{The specific functions}

The specific functions we use to prove \thmref{thm:at_most_8} using \lemref{lem:upper} are described in the following two lemmas. The proofs, like many others in this paper, are computer-assisted. The rigorous verification that the functions we found satisfy the properties we want was done in \texttt{SageMath} \cite{sagemath}. For each computer-assisted proof, there is a corresponding using \texttt{Sage} file performing the verification as well as a text file containing its output. These are included as ancillary files with the arXiv version of this paper, at 
\begin{center}
\url{https://arxiv.org/src/2111.14699v4/anc}.
\end{center}

\begin{lem} \label{lem:f0}
There exists an admissible function $f$ such that $f(x)\leq 0$ for every $x \in \RR$, $\what{f}(r(\lambda))\geq 0.673429$ for every $\lambda \in [1.04,1.857]$, $\what{f}(r(\lambda))\geq 0$ for every $\lambda \geq 1.857$, and \[\frac{1}{0.673429}(\calI - \what{f}(i/2)) \leq 8.957\]
for $g=3$.
\end{lem}
\begin{proof}
Let $a = 1.04$ and $b=1.857$. We choose $f$ of the form $f(x) = u(x^2)e^{-x^2/2}$ with $u(y) = \sum_{n=0}^N s_n L_n^{-1/2}(y)$ for some coefficients $s_n$, where the $L_n^{-1/2}(y)$ are Laguerre polynomials, so that $\what{f}(\xi) = v(\xi^2)e^{-\xi^2/2}$ with $v(y) = \sum_{n=0}^N (-1)^n s_n L_n^{-1/2}(y)$. The coefficients $s_n$ are chosen in such a way that $u$ has a simple root at $0$ and double roots at ${42.26}$ and ${79.78}$,
$v$ has double roots at ${4.57}$, ${11.43}$, ${21.45}$, ${35.34}$, and ${54.4}$, and $v(a-1/4)=1$. This last condition is a normalization which ensures that
\[
\what{f}(r(\lambda))=\what{f}(\sqrt{\lambda-1/4})=v(\lambda-1/4)e^{-(\lambda-1/4)/2}
\]
is positive at $\lambda=a$. 

These conditions on $u$ and $v$ define a system of linear equations for the coefficients $s_n$ which has a unique solution when the number $N+1$ of coefficients is equal to $16$, the number of equations. The coefficients $s_n$ are solved for over the rational numbers, yielding an exact solution. The formal verification that the resulting polynomials $u$ and $v$ have roots of the prescribed orders at the prescribed places was done using the program \texttt{verify0.sage}, whose output can be read in the file \texttt{output0.txt}.

The program \texttt{verify0.sage} also proves that $\what{f}(r(\lambda)) \geq  0.673429$ for every $\lambda \in [a,b]$ as follows. It first checks that $2v'-2v$ has a unique zero between $a-1/4$ and $b-1/4$ using Sturm's theorem (implemented in \texttt{PARI/GP} \cite{PARI2}, which \texttt{Sage} calls). This implies that $\what{f}$ has a unique critical point between $r(a)$ and $r(b)$ since
\[
{\what{f}}'(\xi) = (2v'(\xi^2)-v(\xi^2)) e^{-\xi^2/2} \xi.
\]
The program then checks that $2v' - v$ is positive at $a-1/4$ and negative at $b-1/4$. This implies that the critical point is a local maximum, and hence that the minimum of $\what{f}$ on the interval $[r(a),r(b)]$ is achieved at the endpoints. To conclude the proof, the program verifies that $\what{f}(r(a)) \geq  0.673429$ and $\what{f}(r(b)) \geq  0.673429$ using interval arithmetic (as implemented in the \texttt{Arb} library \cite{Johansson}).

In order to show that $\what{f}(r(\lambda))\geq 0$ for every $\lambda \geq b$, we need to make sure that $\what{f}$ does not change sign on the interval $[r(b),\infty)$ since we already know that $\what{f}(r(b)) \geq  0.673429 \geq 0$. For this, we use Sturm's theorem to check that $v$ has only $5$ distinct zeros in $[b-1/4,\infty)$. As we already checked that these are double zeros, we conclude that $v$ does not change sign on $[b-1/4,\infty)$ and neither does $\what{f}$ on $[r(b),\infty)$. Similarly, we can count the number of zeros of $u$ in $[0,\infty)$ to make sure that it does not change sign, and verify that $u''$ is negative at any of the double zeros so that $u$ and hence $f$ is non-positive on $[0,\infty)$.

The last thing to check is that $\calI - \what{f}(i/2) \leq 0.673429 \cdot 8.957$. For this, we can estimate $\what{f}(i/2)$ using interval arithmetic as before. To get an upper bound on
\[
\calI = 2(g-1)\int_0^\infty r \what{f}(r) \tanh(\pi r) dr = 4\int_0^\infty r \what{f}(r) \tanh(\pi r) dr,
\]
we split the integral into two parts. Indeed, the \texttt{Arb} library can provide rigorous bounds on definite integrals, but not improper ones. We thus use \texttt{Arb} to compute an upper bound on
\[
4\int_0^{100} r \what{f}(r) \tanh(\pi r) \,dr
\]
and then estimate
\[
4\int_{100}^\infty r \what{f}(r) \tanh(\pi r) \,dr \leq 4\int_{100}^\infty r \what{f}(r) \,dr = 4\int_{100}^\infty r \, v(r^2) e^{-r^2/2} \,dr
\]
since $\what{f}$ is non-negative on $[100,\infty) \subset [r(b),\infty)$.
The advantage is that since $r \, v(r^2)$ is an odd polynomial, the function $r \, v(r^2) e^{-r^2/2}$ has an explicit primitive that can be found using integration by parts. One of these primitives has the form $V(r) = p(r)e^{-r^2/2}$ for some polynomial $p$, so that $\lim_{r \to \infty} V(r) = 0$ and hence
\[
4\int_{100}^\infty r \, v(r^2) e^{-r^2/2} \,dr = -4V(100).
\]

Combining the resulting estimates yields
\[
\frac{1}{0.673429}(\calI - \what{f}(i/2)) \leq 8.95625495601736 < 8.957
\]
as required.
\end{proof}

\begin{rem}
The condition $u(0)=0$ is not used anywhere in the above proof. We impose it only because it increases the stability of the program we used to find $f$ and seems to be optimal anyway.
\end{rem}

\begin{lem} \label{lem:f1}
There exists an admissible function $f$ such that $f(x)\leq 0$ for every $x \in \RR$, $\what{f}(r(\lambda))\geq 0.447759$ for every $\lambda \in [1.857,2.71]$, $\what{f}(r(\lambda))\geq 0$ for every $\lambda \geq 2.71$, and \[
\frac{1}{0.447759}(\calI - \what{f}(i/2)) \leq 8.967
\]
for $g=3$.
\end{lem}
\begin{proof}
The construction and proof is similar as in \lemref{lem:f0}. Using the same notation as before, the functions $f(x) = u(x^2)e^{-x^2/2}$ and $\what{f}(\xi) = v(\xi^2)e^{-\xi^2/2}$ are chosen such that $u$ vanishes at $0$ and has double roots at $19.68$ and $50.7$, and $v$ has double roots at $7.87$, $13.78$, $19.7$, $30.22$, $41.93$, and $57.93$, and satisfies $v(1.857-1/4)=1$. The coefficients $s_n$ required as well as the resulting polynomials $u$ and $v$ are given in the file \texttt{verify1.sage}, which formally verifies the required properties and computes that 
\[
\frac{1}{0.447759}(\calI - \what{f}(i/2)) \leq 8.96613860725115 < 8.967.
\]
Its output can be read in \texttt{output1.txt}.
\end{proof}

Together with the results from \secref{sec:sevennec}, the above estimates imply that $m_1(S) \leq 8$ for every closed hyperbolic surface $S$ of genus $3$, as follows.

\begin{proof}[Proof of \thmref{thm:at_most_8}]
Let $S$ be a closed hyperbolic surface of genus $3$. If $\lambda_1(S) \leq 1.04$, then $m_1(S) \leq 6$ according to \thmref{thm:not_too_large}. If $\lambda_1(S) \in [1.04,1.857]$, then \lemref{lem:upper} applied with the function $f$ from \lemref{lem:f0} shows that $m_1(S) \leq 8.957$, and hence  $m_1(S) \leq 8$ since $m_1(S)$ is an integer. Similarly, if $\lambda_1(S) \in [1.857,2.71]$, then \lemref{lem:upper} applied with the function $f$ from \lemref{lem:f1} shows that $m_1(S) \leq \lfloor 8.967 \rfloor = 8$. By Ros's inequality \eqref{eq:ros}, $\lambda_1(S)$ cannot exceed $2.71$, so this proves the result.
\end{proof}

Our next goal is to prove that the inequality from \thmref{thm:at_most_8} is an equality for the Klein quartic. Part of the argument applies to all highly symmetric surfaces called kaleidoscopic, which we discuss first.

\section{Kaleidoscopic surfaces}

For a surface $S$ with several symmetries, one way to prove that $m_1(S)$ (or the multiplicity of any eigenvalue) is large is to study the representation theory of its isometry group $\Isom(S)$. For any eigenfunction $f$ of the Laplacian on $S$ and any $h \in \Isom(S)$, we have that $f \circ h^{-1}$ is an eigenfunction corresponding to the same eigenvalue. That is, the group $\Isom(S)$ acts on the eigenspaces of $\Delta$, and this action is linear. By Maschke's theorem, any eigenspace $\calE$ decomposes into a direct sum of irreducible representations of $\Isom(S)$. If one can show that the irreducible representations of $\Isom(S)$ of small dimension cannot occur in $\calE$, then this shows that the dimension of $\calE$ is greater than or equal to the next smallest dimension among irreducibles. This strategy was used in \cite{Jenni} to prove lower bounds on $m_1$ for the Bolza surface and in \cite{Cook} for the Klein quartic. We generalize and simplify Cook's arguments here. Specifically, we will rule out $1$-dimensional representations for all small enough eigenvalues on kaleidoscopic surfaces, of which the Bolza surface and the Klein quartic are examples. 

\begin{rem} \label{rem:real_vs_complex}
Since we are taking the Laplacian to act on real-valued functions, it is the real representation theory of $\Isom(S)$ which is relevant here. Of course, $\Delta$ can be extended to complex-valued functions by linearity, but there is no advantage in doing so. For instance, Courant's nodal domain theorem only holds for real-valued functions (because $0$ disconnects $\RR$ but not $\CC$). 

On the other hand, character tables and irreducible representations of finite groups are usually listed for complex representations rather than real ones. An irreducible complex representation of dimension $d$ of a finite group $G$ is either realizable over $\RR$ (where it remains irreducible) or else its realification of dimension $2d$ (obtained by restricting the scalars to $\RR$) is irreducible. Moreover, every irreducible real representation of $G$ arises in one of these two ways \cite[p.108]{Serre}. That is, every irreducible real representation of $G$ of dimension $d$ is either the restriction to the real locus of an irreducible complex representation of dimension $d$ or the realification of an irreducible complex representation of dimension $d/2$. Furthermore, one can compute which case occurs from the character table of the group $G$ using \cite[Proposition 39]{Serre}.
\end{rem}

The following definition is taken from \cite{Broughton}.

\begin{defn}
Let $p \leq q \leq r$ be integers. A \emph{$(p,q,r)$-kaleidoscopic surface} is a connected Riemannian surface $S$ of constant curvature that admits a tiling $\calT$ by geodesic triangles with interior angles $\frac{\pi}{p}$, $\frac{\pi}{q}$, and $\frac{\pi}{r}$ such that for every side $e$ of a triangle in $\calT$, there is an isometry of $S$ that acts as a reflection along $e$.
\end{defn}

Note that the dual graph of any triangulation of a connected surface is connected. It follows that if $S$ is kaledoscopic, then for any two triangles $T_1, T_2 \in \calT$, there is an $h \in \Isom(S)$ such that $h(T_1)=T_2$. Indeed, given any path between $T_1$ and $T_2$ in the dual graph, the composition of the reflections corresponding to the edges along this path takes $T_1$ to $T_2$. Also note that $h$ may not be unique if the triangles in $\calT$ are isoceles.   

\begin{rem}
A $(p,q,r)$-kaleidoscopic surface $S$ with $p=2$ is called \emph{Platonic}, because it admits a tiling by regular $r$-gons (and a dual tiling by $q$-gons) such that the isometry group of $S$ acts transitively on the flags (nested triples of vertex, edge, and face) of either tiling.
\end{rem}

We now show that $1$-dimensional real representations cannot appear for eigenvalues that are small enough with respect to either their size or their index on kaleidoscopic surfaces. 

\begin{prop} \label{prop:kalei}
Let $S$ be a $(p,q,r)$-kaleidoscopic hyperbolic surface. Then no $1$-dimensional real representation of $\Isom(S)$ can occur for any eigenvalue of the Laplacian in the interval $(0,\lambda_0(F))$ where $F$ is a hyperbolic disk of area $A=2\pi r \left( 1 - \frac1p - \frac1q -\frac1r \right)$, nor for any eigenvalue $\lambda_j$ with $1 \leq j \leq \area(S)/A-2$.
\end{prop}
\begin{proof}
Let $f$ be a non-constant eigenfunction of the Laplacian spanning a $1$-dimensional real representation of $\Isom(S)$. Then for every $h \in \Isom(S)$, we have that $f \circ h^{-1} = \pm f$. In particular, the nodal set $\calZ(f)=f^{-1}(0)$ is invariant under $\Isom(S)$. We will use the symmetries of the kaleidoscopic tiling $\calT$ of $S$ by $(p,q,r)$-triangles to show that there is a nodal domain of $f$ whose area is not too large.

Let $T$ be any triangle in $\calT$ and let $C = \calZ(f)\cap T$. Since $f$ is non-constant, it has at least two nodal domains so that $\calZ(f)$ is non-empty. As $\Isom(S)$ acts transitively on the triangles in $\calT$ and preserves the nodal set, $C$ is non-empty as well.

We consider several possibilities for what $C$ can look like. If $C$ is disjoint from the sides of $T$, then there is a nodal domain contained in $T$ since $\calZ(f)$ is a union of smooth circles. If $C$ connects one side $s$ of $T$ to itself, then by reflecting across $s$ we see that there is a nodal domain contained in the union of two triangles in $\calT$. If $C$ connects two distinct sides of $T$, then by repeated reflections we see that there is a nodal domain contained in the union of all the triangles that share the same vertex as those two sides. If none of these things occur, then $C$ must be contained in $\partial T$. However, if $\calZ(f)$ contains any segment of a side $s$ of $T$, then it must contain the entire side as well as the closed geodesic that contains it. Indeed, if $L$ is a nodal line with a segment along $s$, then the reflection of $L$ across $s$ is a nodal line that intersects $L$ tangentially. By Cheng's theorem, this implies that $L$ is equal to its reflection, hence is contained in the locus of reflection, which is a union of closed geodesics. Since nodal lines are smooth circles, $L$ is equal to one of these closed geodesics. We are back to the previous case because then $C$ connects two distinct sides of $T$.

In all of these cases, there is a nodal domain $U$ of $f$ contained in the (simply connected) union of at most $2r$ triangles in $\calT$ (this is the number of triangles around a vertex of angle $\pi/r$). The area of $U$ is thus at most $2r$ times the area of $T$, or
\[
A=2\pi r \left( 1 - \frac1p - \frac1q -\frac1r \right).
\]
Let $F$ be a hyperbolic disk of area $A$. Then $\lambda = \lambda_0(U) \geq \lambda_0(F)$
by the Faber--Krahn inequality \cite[p.87]{Chavel}. 

Moreover, the number of nodal domains of $f$ is greater than or equal to the size of the orbit of $U$ under the group $G$ generated by the reflections in the sides of triangles in $\calT$ (which may be smaller than $\Isom(S)$ a priori). The size of this orbit is equal to the order of $G$ divided by the order of the stabilizer of $U$ in $G$. Depending on whether $U$ is contained in a single triangle, a pair of triangles, or a fan around a vertex, its stabilizer has order at most $1$, $2$, $2p$, $2q$, or $2r$. In any case, its order is at most $2r$. Now $|G|$ is at least the number of triangles in $\calT$, or $\area(S) / (A/2r)$, so that the orbit of $U$ has size at least $\area(S)/A$. If $\lambda=\lambda_j(S)$, then Courant's nodal domain theorem tells us that $\area(S)/A \leq j+1$.
\end{proof}

\begin{rem} \label{rem:kalei}
Artamoshin's inequality \eqref{eq:arta} yields the estimate
\[
\lambda_0(F) \geq \frac14+ \frac{\pi^2}{16 \arcsinh^2\left( \sqrt{\frac{r}{2} \left( 1 - \frac1p - \frac1q -\frac1r \right)} \right)}.
\]
\end{rem}

\section{The Klein quartic} \label{sec:klein}

The Klein quartic $K$ is the hyperbolic surface of genus $g=3$ with the largest isometry group; it achieves Hurwitz's upper bound of $168(g-1)$. It is a $(2,3,7)$-kaleidoscopic surface as defined in the previous section. A fundamental domain for $K$ together with side pairings is given in \figref{fig:klein}. We refer the reader to \cite{KarcherWeber} for more on the geometry of $K$.

The goal of this section is to prove the following theorem.

\begin{thm} \label{thm:klein}
The Klein quartic satisfies $m_1(K) = 8$.
\end{thm}

Our proof is largely inspired by work of Cook, who showed that $m_1(K) \in \{ 6, 7, 8\}$  \cite[Corollary 4.5]{Cook}.

\subsection{One-dimensional representations}\label{sec:one_dim_reps_Klein}

As explained in the previous section, the real representation theory of $\Isom(K)$ is re\-le\-vant for studying the multiplicity of eigenvalues. The group $\Isom(K)$ is isomorphic to $\PGL(2,\ZZ/7\ZZ)$, which has two irreducible complex representations of dimension $1$, three of dimension $6$, two of dimension $7$, and two of dimension $8$. These are listed in \cite[Appendix B]{Cook}. They can all be realized over $\RR$. The character table can also be found under the entry \texttt{PGL(2,7)} in \cite{Dokchitser}, or it can be obtained by running the command \texttt{PGL(2,7).character\_table()} in \texttt{SageMath}.

We start by applying \propref{prop:kalei} to rule out $1$-dimensional representations for small enough eigenvalues.

\begin{cor} \label{cor:1dim}
The $1$-dimensional real representations of $\Isom(K)$ do not occur for any Laplacian eigenvalue of $K$ in the interval $(0,7.85]$. 
\end{cor}

\begin{proof}
Since $K$ is a $(2,3,7)$-kaleidoscopic surface, \propref{prop:kalei} shows that $1$-dimensional real representation of $\Isom(K)$ can only occur for non-trivial eigenvalues larger than or equal to
\[
 \frac14+ \frac{\pi^2}{16 \arcsinh^2\left( \sqrt{\frac{7}{2} \left( 1 - \frac12 - \frac13 -\frac17 \right)} \right)} \approx 7.854527\ldots > 7.85.
\]
\end{proof}

We can then deduce that any eigenvalue in that interval must occur with high multiplicity.

\begin{cor} \label{cor:mult_6}
Any eigenvalue of the Laplacian on $K$ in the interval $(0,7.85]$ occurs with multiplicity at least $6$ and if its multiplicity is more than $8$, then it is at least $12$.
\end{cor}
\begin{proof}
By \corref{cor:1dim}, the $1$-dimensional real representations of $\Isom(K)$ do not occur in the eigenspace $\calE$ corresponding to any eigenvalue $\lambda \in (0,7.85]$. The eigenspace $\calE$ can therefore be decomposed into irreducible real representations of dimensions $6$, $7$ or $8$ since these are the dimensions of the remaining irreducible real representations of $\Isom(K)$. In particular, $\calE$ has dimension at least $6$. If $\dim \calE > 8$, then $\calE$ contains at least two irreducible representations of dimension at least $6$ each, for a total of at least $12$. 
\end{proof}

The above applies in particular to $\lambda_1(K)$, which leads to a lower bound on the latter that we will use in the next subsection. 

\begin{cor}  \label{cor:init_lower}
The Klein quartic satisfies $m_1(K) \geq 6$ and \[\lambda_1(K) > \frac14+ \left(\frac{\pi}{4\arcsinh\left(\sqrt{2}\right)} \right)^2 \approx 0.719512\ldots\]
\end{cor}
\begin{proof}
By \cite{Ros}, we have $\lambda_1(K) \leq 2(4 - \sqrt{7}) \approx 2.708497 < 7.85$, so the first inequality follows from \corref{cor:mult_6}. Since $6> 2g-1$ where $g=3$ is the genus of $K$, the second inequality follows from the first and \thmref{thm:not_too_large}.
\end{proof}

Recall that $\lambda_1(K)\approx 2.6779$ according to numerical calculations from \cite[Section 4.3]{KMP}. The bound from \corref{cor:init_lower} is just an initial estimate that we need to use for our linear programming bounds. The better estimate $\lambda_1(K) \geq 2.575$ follows from Proposition \ref{prop:less6} below.

\subsection{Linear programming bounds}

In order to prove that $m_1(K)=8$, we use the Selberg trace formula. In view of the previous subsection, it suffices to prove the following three propositions, where $\Delta_K$ denotes the Laplacian on the Klein quartic $K$.

\begin{prop} \label{prop:less6}
The number of eigenvalues of $\Delta_K$ in $[0.71,2.575]$ is strictly less than $6$ counting multiplicity.
\end{prop}

\begin{prop} \label{prop:less12}
The number of eigenvalues of $\Delta_K$ in $[2.575,5.5]$ is strictly less than $12$ counting multiplicity.
\end{prop}

\begin{prop} \label{prop:more7}
The number of eigenvalues of $\Delta_K$ in $[2.575,5.5]$ is strictly more than $7$ counting multiplicity.
\end{prop}

Assuming these for a moment, we can prove that $m_1(K)=8$ as follows.
\begin{proof}[Proof of \thmref{thm:klein}]
By \corref{cor:mult_6}, any eigenvalue in $(0,7.85]$ occurs with multiplicity at least $6$. Together with \propref{prop:less6}, this implies that there are no eigenvalues at all in $[0.71,2.575]$. Combined with the bound $\lambda_1(K)>0.71$ from \corref{cor:init_lower}, this shows that $\lambda_1(K)>2.575$. Then \propref{prop:less12} and \corref{cor:mult_6} tell us that there is at most one distinct eigenvalue in the interval $[2.575,5.5]$ and that its multiplicity is either $6$, $7$, or $8$. Finally, \propref{prop:more7} rules out multiplicities $6$ and $7$.
\end{proof}

A consequence of the above results is that the next distinct eigenvalue $\lambda_9(K)$ is at least $5.5$ and indeed, numerical calculations from \cite[Table D.1]{Cook} indicate that $\lambda_9(K) \approx 6.61848$. 

To prove Propositions \ref{prop:less6} and \ref{prop:less12}, we use the strategy from \lemref{lem:upper}, but this time we can include information about the length spectrum of $K$. It turns out that just using the set of systoles in $K$ for the set $P$ of primitive closed geodesics in the lemma is not enough to obtain the inequalities we are after; we also need to include the second shortest closed geodesics on the Klein quartic. These were determined numerically in \cite[Table C.1]{Vogeler1} and \cite[Theorem 31]{DT}. That they are the second shortest closed geodesics can be verified by a computer check as in \cite[Theorem 31]{DT} (the numerical calculation can be made rigorous since one can find exact generators for the uniformizing Fuchsian group of $K$), though we will not need this. What we need is a formula for the length of these second shortest curves and a lower bound for their multiplicity.

\begin{lem} \label{lem:second_shortest}
There are at least 28 unoriented primitive closed geodesics of length
\[
6 \arccosh\left(8\cos^4(\pi/7)-6\cos^2(\pi/7)+1\right) \approx 5.208017\ldots
\]
on the Klein quartic.
\end{lem}
\begin{proof}
The Klein quatic is tiled by 24 regular heptagons with interior angles $2\pi / 3$. Take any geodesic segment $\sigma$ that joins the midpoints of two sides that are separated by one other side in such a heptagon. By extending $\sigma$ in both directions, we obtain a primitive closed geodesic that crosses a total of $6$ heptagons in the same way (joining the midpoints of two sides separated by another). One example of such a geodesic is depicted in \figref{fig:klein}. The fact that this particular example closes up implies that all others do since the isometry group of the Klein quartic acts transitively on flags of heptagons in the tiling, hence on geodesic segments with the same description as $\sigma$. There are $7$ geodesic segments to choose from per heptagon and $24$ heptagons, while each closed geodesic obtained consists of $6$ segments, for a total of $7\times 24 / 6 = 28$ closed geodesics of this kind.

\begin{figure}[htp]
\centering
\includegraphics{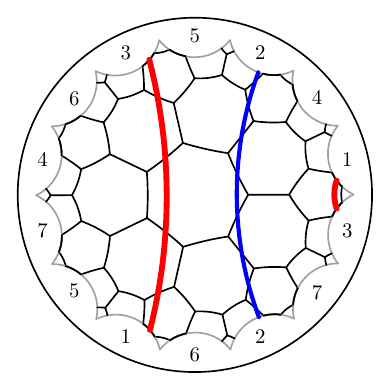}
\caption{A fundamental domain for the Klein quartic shown with its tessellation by regular heptagons with interior angles $2\pi / 3$. The numbers indicate the side pairings. One of the closed geodesics from \lemref{lem:second_shortest} is drawn in red and a systole is drawn in blue.}
\label{fig:klein}
\end{figure}

The length of these closed geodesics is equal to $6$ times the length of $\sigma$. Note that $\sigma$ forms an isoceles triangle with the center of its heptagon $H$, with two sides of length equal to the inner radius $r$ of $H$ (the second shortest side of a $(2,3,7)$-triangle), and angle $4\pi/7$ at the center of $H$. Trigonometric formulas for right-angled triangles \cite[p.454]{Buser} yield 
\[r = \arccosh(\cos(\pi/3)/\sin(\pi / 7)) = \arccosh(1/(2\sin(\pi / 7)))\]
and
\[
\ell(\sigma)/2 = \arcsinh(\sin(2\pi/7) \sinh(r)).
\]
Since $\sinh(\arccosh(y)) = \sqrt{y^2 - 1}$, we have $\sinh(r)=\sqrt{\frac{1}{4\sin^2(\pi/7)}-1}$. Combining this with $\sin(2\pi / 7)= 2 \sin(\pi/7)\cos(\pi/7)$ yields the folmula
\[
L = 6\, \ell(\sigma) = 12 \arcsinh\left(\cos(\pi/7)\sqrt{1-4\sin^2(\pi/7)}\right)
\]
for the length $L$ of the resulting geodesics. Using the identity $\cosh(2x)=2\sinh^2(x)+1$, we get that $2\arcsinh(y) = \arccosh(2y^2+1)$ for $y>0$, and hence
\begin{align*}
L &= 6 \arccosh\left(2\cos^2(\pi/7)(1-4\sin^2(\pi/7))+1\right)\\
& = 6 \arccosh\left(2\cos^2(\pi/7)(4\cos^2(\pi/7)-3)+1\right)\\
&= 6 \arccosh\left(8\cos^4(\pi/7)-6\cos^2(\pi/7)+1\right).\qedhere
\end{align*}
\end{proof}

As for the shortest closed geodesics in $K$, they were identified in \cite{Schmutz}. Each one consists of $8$ geodesic segments joining midpoints of adjacent sides of heptagons in the tiling. From this description, we see that there are $7 \times 24 / 8 = 21$ such closed geodesics. If $\sigma$ is any segment joining the midpoints of adjacent sides of a heptagon, then
\[
\ell(\sigma) / 2 = \arcsinh(\sin(\pi/7)\sinh(r))
\]
where $r$ is as in the above proof. We have
\[
\sin(\pi/7)\sinh(r) =  \sqrt{\frac14- \sin^2(\pi / 7)}
\]
so that
\[
\sys(K) = 8 \ell(\sigma) = 16 \arcsinh \left( \sqrt{\frac14- \sin^2(\pi / 7)} \right)\approx 3.935946\ldots
\]
Using the identity $2\arcsinh(y) = \arccosh(2y^2+1)$ for $y>0$, the above simplifies to
\[
\sys(K) = 8 \arccosh\left( \frac32 - 2 \sin^2(\pi/7) \right)=8 \arccosh\left( \frac12 +  \cos(2\pi/7) \right).
\]
This formula is attributed to Schmutz in \cite{Quine} and \cite{Cook}, but we could not find it in \cite{Schmutz}. The surface $K$ coincides with what Schmutz calls $T(x|z)$ according to Remark (i) on page 624 since there is only one $(2,3,7)$-triangle surface in genus $3$. Schmutz approximates its systole as $2\arccosh(3.648)$ on page 567 and as $3.9359$ on page 622. That the eight-step geodesics described above are the systoles in $K$ can alternatively be proved by a computer check \cite[Theorem 24]{DT}.

We can now prove the first two propositions stated above.

\begin{proof}[Proof of \propref{prop:less6}]
We apply \lemref{lem:upper} with $S=K$, $a = 0.71$, $b=2.575$ and the function $f(x) = u(x^2) e^{-x^2/2}$ with Fourier transform $\what{f}(x) = v(x^2) e^{-x^2/2}$ specified by the following conditions. We require that $u$ satisfies $u(0)=1$, $u(3.9359^2)= -716\cdot 10^{-11}$, and that it has a double zero at $58.59$, and that $v$ has double zeros at
\[
2.6, 6.54, 11.62, 18.9, 25.71, 37.98, \text{ and } 53.32.
\]
The required coefficients and resulting polynomials are given in the files \texttt{verify2.sage} and \texttt{output2.txt}. 

To apply the lemma, we need to verify that $u(x^2) \leq 0$ if $x\geq \sys(K)$. Here we use the estimate $\sys(K) > 3.9359$ and verify that $u(x^2)$ has no zeros in $[3.9359,\infty)$ besides the one at $\sqrt{58.49}$. Since $u(x^2)$ is negative at $3.9359$ and does not change sign afterwards, it is non-positive on $[\sys(K),\infty) \subset [3.9359,\infty)$.

Note that $0<a = 0.71  \leq \lambda_1(K)$ by \corref{cor:init_lower}, as required to apply \lemref{lem:upper}. We then check that $v$ is non-negative on $[a-1/4, \infty)$ by counting its zeros there and verifying that $v(a-1/4)>0$. We also verify that $\what{f}$ is either monotone or has a unique local maximum in $[r(a),r(b)]$ as in in the proof of \lemref{lem:f0} in order to determine its minimum $c$ on that interval (achieved at an endpoint).

The integral term $\calI$ and $-\what{f}(i/2)$ are bounded above in the same way as in \lemref{lem:f0}. For the geometric term, we use the set $P$ of shortest and second shortest closed geodesics in $K$, and use interval arithmetic to compute the corresponding summands. This is why we needed exact formulas for the lengths, so that this last calculation is rigorous. 

The resulting upper bound for the number of eigenvalues between $a$ and $b$ is
\[
\frac{1}{c}\left(\calI + \calG -\what{f}(i/2)\right) \leq 4.47792022244533 < 6
\]
as required.
\end{proof}

\begin{proof}[Proof of \propref{prop:less12}]
We apply \lemref{lem:upper} with $S=K$, $a = 2.575$, $b=5.5$ and $P$ as before. By \propref{prop:less6} and \corref{cor:mult_6}, we know that $\lambda_1(K) \geq a$. The function $f(x) = u(x^2) e^{-x^2/2}$ and its Fourier transform $\what{f}(x) = v(x^2) e^{-x^2/2}$ are determined by requiring that $u$ satisfies $u(3.9359^2)=-1$ and that it has a double zero at $8.43$, and that $v$ has a simple zero at the origin and double zeros at
\[
9.19,14.9,26.68,88.17,89.17,92.35,100, \text{ and }115.16.
\]
The required coefficients are given in the file \texttt{verify3.sage}, which performs all the necessary checks and gives the upper bound
\[
\frac{1}{c}\left(\calI + \calG -\what{f}(i/2)\right) \leq 11.7388106033274 < 12
\]
for the number of eigenvalues between $a$ and $b$. The output of the program can be read in \texttt{output3.txt}.
\end{proof}

The last proposition will use the following variant of \lemref{lem:upper}.

\begin{lem}  \label{lem:lower}
Let $S$ be a closed hyperbolic surface of genus $g \geq 2$. Suppose that $0 < a \leq \lambda_1(S)$ and $a<b$, and let $f$ be an admissible function such that $\what{f}(r(\lambda)) \leq c $ for some $c>0$ and every $\lambda \in [a,b]$, $\what{f}(r(\lambda)) \leq 0$ for every $\lambda \geq b$,  and $f(x)\geq 0$ for every $x \geq \sys(S)$. Then the number of eigenvalues in $[a,b]$ counting multiplicity is at least
\[
\frac{1}{c}\left( 2(g-1) \int_0^\infty r \what{f}(r) \tanh(\pi r)\,dr + \frac{1}{\sqrt{2\pi}}\sum_{\gamma \in P} \frac{\ell(\gamma) f(\ell(\gamma))}{\sinh(\ell(\gamma)/2)} - \what{f}(i/2) \right)
\]
for any set $P$ of unoriented primitive closed geodesics in $S$. 
\end{lem}
\begin{proof}
We have
\[
c \cdot \left|\sigma(S) \cap[a,b] \right|   \geq \sum_{\lambda \in \sigma(S) \cap[a,b]} \what{f}(r(\lambda)) \\
 \geq \calS  - \what{f}(i/2)
 = \calI + \calG - \what{f}(i/2)
\]
where the notation is as in subsection \ref{subsec:STF}.

Since $f$ is non-negative on the length spectrum of $S$, we have
\[
\calG = \frac{1}{\sqrt{2\pi}}\sum_{\gamma \in \calC(S)} \frac{\ell(\gamma) f(\ell(\gamma))}{2\sinh(\ell(\gamma)/2)} \geq \frac{1}{\sqrt{2\pi}}\sum_{\gamma \in P} \frac{\ell(\gamma) f(\ell(\gamma))}{\sinh(\ell(\gamma)/2)} 
\]
where the factor of $2$ is because we pass from oriented closed geodesics to unoriented ones. Dividing by $c$ gives the desired inequality.
\end{proof}

We use this to show that $\Delta_K$ has more than $7$ eigenvalues in the interval $[2.575,5.5]$.

\begin{proof}[Proof of \propref{prop:more7}]
We apply \lemref{lem:upper} with $S=K$, $a = 2.575$, $b=5.5$ and $P$ the set of shortest and second shortest closed geodesics in $K$. We know that $\lambda_1(K) \geq a$ as noted before. We then take $f(x) = u(x^2) e^{-x^2/2}$ with Fourier transform $\what{f}(x) = v(x^2) e^{-x^2/2}$ such that $u$ has a double zero at $2.8$ and $v$ satisfies $v(0)=1$, $v(b-1/4)=-2/10^{12}$, and has double zeros at $9$, $15$, $23$ and $34$. The required coefficients are given in the file \texttt{verify4.sage}, which checks that $u$ is non-negative everywhere in $[0,\infty)$, that $v$ is non-positive on $[b-1/4,\infty)$, and that $\what{f}$ is decreasing on $[r(a),r(b)]$, so that $c=\what{f}(r(a))$ is an upper bound for $\what{f}$ on that interval.  The output of the program can be read in \texttt{output4.txt}.

To bound the integral term from below, we can use
\[
\calI = 4 \int_0^\infty r \what{f}(r) \tanh(\pi r)\,dr \geq  4 \int_0^{100} r \what{f}(r) \tanh(\pi r)\,dr + 4 \int_{100}^\infty r \what{f}(r)\,dr 
\]
since $\what{f}(r) \leq 0$ for every $r \in \left[\sqrt{b-1/4},\infty\right) \supset [100,\infty)$. The first term is estimated with interval arithmetic and the second term is computed using integration by parts.

The resulting lower bound is 
\[
\frac{1}{c}\left(\calI + \calG -\what{f}(i/2)\right) \geq 7.70518139471331  > 7
\]
for the number of eigenvalues between $a$ and $b$.
\end{proof}

This completes the proof that $m_1(K) = 8$, which was the second half of \thmref{thm:main} stated in the introduction.

\section{Genus two}

The goal of this section is to prove \thmref{thm:genus2}, which states that $m_1(S) \leq 6$ for every closed hyperbolic surface $S$ of genus $2$. By \thmref{thm:not_too_large}, we have $m_1(S) \leq 2g = 4$ whenever $\lambda_1(S) \leq  1.672643 < \frac14+\eps_2$. We can therefore assume that $\lambda_1(S) \geq 1.67$.

We will also need an upper bound on $\lambda_1(S)$. As observed in \cite{ElSoufi}, the inequality from \cite{YangYau} combined with the fact that every Riemann surface of genus $g$ admits a meromorphic function of degree at most $\left\lfloor\frac{g+3}{2} \right\rfloor$ yields that
\begin{equation}  \label{eq:yang_yau}
\lambda_1(S) \leq \frac{8\pi}{\area(S)} \left\lfloor\frac{g+3}{2} \right\rfloor
\end{equation}
for every closed orientable Riemannian surface $S$ of genus $g$. In genus $2$, this gives
\[
\lambda_1(S) \leq \frac{16\pi}{\area(S)},
\]
which turns out to be sharp \cite{NayataniShoda} as conjectured in \cite{Jakobson}. There is a $1$-dimensional family of spherical metrics with cone points realizing the equality, one of which is in the conformal class of the Bolza surface. For closed hyperbolic surfaces of genus $2$, the above inequality reduces to $\lambda_1(S) \leq 4$, as these have area $4\pi$. 

In genus $3$, the inequality of Ros \cite{Ros} used earlier improves upon \eqref{eq:yang_yau}. Ros's arguments were generalized to higher genus in \cite{Karpukhin}, yielding an improvement upon \eqref{eq:yang_yau} in all but finitely many genera. If we restrict to hyperbolic surfaces, then the inequality \eqref{eq:cheng_gage} of Cheng and Gage is asymptotically smaller than the one in \cite{Karpukhin} as the genus tends to infinity.

In view of the above, \thmref{thm:genus2} is reduced to showing that $m_1(S) \leq 6$ whenever $\lambda_1(S) \in [1.67,4]$. We will prove this with linear programming in a similar way as in genus $3$.

\begin{prop}
If $S$ be a closed hyperbolic surface of genus $2$ such that $\lambda_1(S) \in [1.67,4]$, then $m_1(S) \leq 6$.
\end{prop}
\begin{proof}
We apply \lemref{lem:upper} on the intervals $[1.67,2.2]$, $[2.2,3.5]$, and $[3.5,4]$. In each case, we choose the polynomials $u$ and $v$ such that $u(0)=0$ and $v(a-1/4)=1$ where $a$ is the left endpoint of the interval. We also impose some double zeros that are listed below. We then check that $u$ is non-positive on $[0,\infty)$ by counting its zeros using Sturm's theorem and by checking the sign of $u''$ at the double zeros. We verify that $v$ is non-negative on $[a,\infty)$ in the same way. Finally, we check that on the interval $[r(a),r(b)]$, the function $\what{f}(x)=v(x^2) e^{-x^2 / 2}$ attains its minimum at the endpoints and estimate the expression on the right-hand side of inequality \eqref{eq:upper_simplified} using interval arithmetic. 

For the interval $[1.67,2.2]$, the double zeros of $u$ are at $26.7$ and $84.69$, and those of $v$ are at $7.06$, $15.89$, and $28.94$. The program \texttt{verify5.sage} makes the required verifications and outputs the upper bound
\[ m_1(S) \leq 5.99284669868481 < 6\]
whenever $\lambda_1(S) \in [1.67,2.2]$.

For the interval $[2.2,3.5]$, the double zeros of $u$ are at $15.57$ and $75.82$, and those of $v$ are at $10.14$, $21.54$, and $37.3$. The program \texttt{verify6.sage} makes the required verifications and outputs the upper bound
\[ m_1(S) \leq 6.81805544044067 < 7\]
whenever $\lambda_1(S) \in [2.2,3.5]$.

For the interval $[3.5,4]$, the double zeros of $u$ are at $11.8$ and $37.39$, and those of $v$ are at $11.85$, $18.03$, $30.21$, and $51.66$. The program \texttt{verify7.sage} makes the required verifications and outputs the upper bound
\[ m_1(S) \leq 6.98887942266988 < 7\]
whenever $\lambda_1(S) \in [3.5,4]$.
\end{proof}

\section{The Bolza surface}  \label{sec:bolza}

The Bolza surface in genus $2$ is obtained by identifying all pairs of opposite sides on the regular hyperbolic octagon with interior angles $\pi/4$ (see \figref{fig:bolza}). We denote this surface by $B$. In \cite{Jenni}, Jenni claims that $m_1(B)=3$ and refers to his thesis \cite{JenniThesis} for the proof. However, Cook \cite{Cook} found a mistake in Jenni's presentation of $\Isom(B)$, which affects the calculation of the irreducible representations of $\Isom(B)$ that are used in a crucial way in the proof. The purpose of this section is to rectify the situation by reproving Jenni's claim.

\begin{thm} \label{thm:Bolza}
The Bolza surface satisfies $m_1(B)=3$.
\end{thm}

The inequality $m_1(B)\geq 3$ is essentially reproduced from \cite{Cook} with minor corrections and some simplifications, while the bound $m_1(B)< 4$ is proved via linear programming.

\subsection{The isometry group}  \label{subsec:isom}

Similarly as for the Klein quartic, the lower bound $m_1(B)\geq 3$ amounts to ruling out irreducible representations of dimensions $1$ and $2$. However, in this case we need explicit matrix representatives for generators under the representations rather than just the character table. This requires writing down a presentation for $\Isom(B)$. One such presentation is given in \cite[Section 3.1]{Cook}:
\[
\Isom(B) = \left< r,s,t,u\;\left|\; \begin{array}{c}
r^8 = s^2 = (rs)^2 = (st)^2 = rt r^3t = e \\ 
ur = r^7u^2,\; u^2r = stu,\; us = su^2,\; ut= rsu
\end{array} \right>\right. .
\]
The generators can be described as follows, where $O$ denotes the regular octagon used to define $B$:
\begin{itemize}
    \item $r$ is a rotation of angle $\pi / 4$ about the center $C$ of $O$,
    \item $s$ is a reflection in one of the main diagonals $D$ of $O$,
    \item $t$ is the reflection across a geodesic segment $L$ that joins the midpoints of two con\-se\-cutive sides of $O$ and is orthogonal to $D$ (there are two choices possible for $L$). More precisely, $t$ is the extension to $B$ of the local reflection across $L$.
    \item $u$ is the extension to $B$ of the rotation of angle $2\pi / 3$ about the center of the equilateral triangle with one vertex at $C$ and the opposite side equal to $L$.
\end{itemize}

That $r$ and $s$ descend to isometries on $B$ is straightforward as they preserve $O$ as well as the gluings. This is less obvious for $t$ and $u$, but it can be shown using a cut-and-paste argument. 

Another approach is to first establish that $B$ is a $(2,3,8)$-kaleidoscopic surface, so that isometries can be described by local data. This can be argued as follows. First observe that the hyperelliptic involution of $B$ is equal to $r^4$, the rotation of order $2$ around the center of $O$. The quotient of $B$ by this involution is a hyperbolic octahedron with the same symmetries as the Euclidean regular octahedron. This can be seen by considering the tiling of $B$ by $16$ equilateral triangles with interior angles $\pi/4$ dual to the tiling shown in \figref{fig:bolza}.

On any genus $2$ surface $S$, the hyperelliptic involution $\iota$ commutes with all the automorphisms, so the subgroup $\langle \iota \rangle$ is normal in $\Aut(S)$ and the quotient group $\Aut(S) / \langle \iota \rangle$ acts by conformal automorphisms on $S/\langle \iota \rangle  \cong \what{\CC}$, preserving the images of the six Weierstrass points (the fixed points of $\iota$). Conversely, any automorphism of $\what{\CC}$ that permutes the cri\-ti\-cal values lifts to $S$. This is true for any hyperelliptic surface as these are completely determined by their ramification points. For the Bolza surface, this means that $\Aut(B) / \langle \iota \rangle \cong \Isom^+(B) / \langle r^4 \rangle $ is isomorphic to the rotation group of the regular octahedron, which is isomorphic to the symmetric group $\sym_4$ (the four pairs of opposite faces can be permuted arbitrarily). It follows that $\Isom^+(B)$ is a central extension of $\sym_4$ by $\ZZ/2\ZZ$, which turns out to be isomorphic to $\GL(2,\ZZ/3\ZZ)$ \cite[p.260]{Broughton2}. 
The reflection $s$ together with $\Isom^+(B)$ generate a group $G$ which acts transitively on the triangles in a tiling $\calT$ of $B$ by $(2,3,8)$-triangles. This tiling is obtained by taking the barycentric subdivision of the sixteen $(4,4,4)$-triangles that map to the faces of the octahedron. Since the triangle $B / G$ does not cover any smaller orbifold, we have $G = \Isom(B)$. In particular, for any $T_1, T_2 \in \calT$ there is an isometry $h \in \Isom(B)$ such that $h(T_1)=T_2$. Furthermore, this isometry is unique because the three sides of a $(2,3,8)$-triangle have different lengths, which determines $h$ on $T_1$. By the identity principle, isometries are uniquely determined by their action on any open set, so $h$ is unique.

Once we know that $r$, $s$, $t$, and $u$ are well-defined, it is elementary to check that the relations hold, as it suffices to check their validity on any triangle in the tiling.

Since $\Isom^+(B)$ is a normal subgroup of index $2$ in $\Isom(B)$, we have
\[
\Isom(B) \cong \GL(2,\ZZ/3\ZZ)\rtimes (\ZZ/2\ZZ).
\]
According to the list of groups of order $96$ on \cite{Dokchitser}, there are three pairwise non-isomorphic such semi-direct products depending on the image of $\ZZ/2\ZZ$ in $\Aut(\GL(2,\ZZ/3\ZZ))$. By feeding the above presentation into the computer algebra system \texttt{GAP} \cite{GAP4} and comparing character tables, we can identify $\Isom(B)$ with the group presented as entry \texttt{C4.3S4} in the list of groups of order at most $500$ in \cite{Dokchitser}. The character table as well as the complex irreducible representations of $\Isom(B)$ can be found in \cite[Appendix A]{Cook}. We will only need to consider the $1$- and $2$-dimensional representations, which all turn out to be real.

\subsection{One-dimensional representations}
The group $\Isom(B)$ admits four distinct $1$-dimen\-sio\-nal real representations, none of which can occur for small enough eigenvalues.

\begin{lem} \label{lem:Bolza_no_1dim}
No $1$-dimensional real representation of the isometry group $\Isom(B)$ can occur for any eigenvalue $\lambda_j(B)$ with $1\leq j \leq 4$.
\end{lem}
\begin{proof}
This follows from \propref{prop:kalei} since $B$ is a $(2,3,8)$-kalei\-dos\-co\-pic surface. The resulting upper bound is
\[
j \leq \frac{\area(B)}{16\pi \left( 1 - \frac12 - \frac 13 - \frac18 \right)} - 2 =\frac{4\pi}{2\pi / 3} - 2 = 4. \qedhere
\]
\end{proof}

\subsection{Two-dimensional representations}

Since all the $1$-dimensional complex representations of $\Isom(B)$ are realizable over $\RR$, there is no irreducible $2$-dimensional real representation that arises by restriction of scalars from one of these. Hence, the $2$-dimensional irreducible real representations come from complex ones of the same dimension. There are two $2$-dimensional irreducible complex representations of $\Isom(B)$, which happen to be realizable over $\RR$ (see \cite[Appendix A]{Cook}).  For each of them, the strategy for showing that it does not occur in the eigenspace for $\lambda_1(B)$ is to find some function $f$ in the representation which is sent to $-f$ by a large number of reflections in $\Isom(B)$. This implies that the nodal set of $f$ contains the locus of fixed points of all of these reflections, and hence that $f$ has several nodal domains.

\begin{lem} No $2$-dimensional irreducible real representation of the isometry group $\Isom(B)$ can occur for any eigenvalue $\lambda_j(B)$ with $j\leq 2$.
\end{lem}

\begin{proof}
Let $\lambda=\lambda_j(B)$ be an eigenvalue of the Laplacian and let $\calE$ be the corresponding eigenspace. Suppose that a $2$-dimensional irreducible real representation $(V, \rho)$ of $\Isom(B)$ appears as a summand in $\calE$. There are two such representations up to conjugation.

The first $2$-dimensional irreducible representation of $\Isom(B)$ is given by
\begin{align*}
\rho(r) = \left(\begin{array}{cc}
    0 & 1  \\
    1 & 0
\end{array} \right), &\quad
\rho(s) = \left(\begin{array}{cc}
    0 & -1  \\
    -1 & 0
\end{array} \right),\\
\rho(t) = \left(\begin{array}{cc}
    -1 & 0  \\
    0 & -1
\end{array} \right), &\quad
\rho(u) = \left(\begin{array}{cc}
    0 & -1  \\
    1 & -1
\end{array} \right)
\end{align*}
with respect to some basis $\{f_1,f_2\}$ for $V$. Consider the function \[f=f_1+f_2=\left(\begin{array}{c} 1 \\ 1 \end{array}\right)\] in these coordinates. Observe that
\[
\rho(r)^k \rho(s) f = - f \quad \text{for} \quad k= 0,\ldots, 7
\]
and
\[
\rho(r)^k\rho(t)\rho(r)^{-k} f = -f  \quad \text{for} \quad k= 0,\ldots, 7.
\]
This implies that $f$ vanishes on the fixed point sets of all these transformations since $\Isom(B)$ acts on functions by precomposition with the inverse. In the octagon $O$, the fixed point set of $s$ is a main diagonal $D$, from which it follows that the fixed point set of $r^{2m}s=r^m s r^{-m}$ is $r^{m}(D)$ for any $m$. The fixed point set $M$ of $rs$ is a geodesic joining the midpoints of two opposite sides of $O$ together with the pair of opposite sides disjoint from the first. The fixed point set of $r^{2m+1}s= r^m rs r^{-m}$ is then $r^m(M)$. As for the fixed point set of $t$, it is equal to $N=L\cup r^2(L) \cup r^4(L) \cup r^6(L)$ where $L$ is as in subsection \ref{subsec:isom}. This is not immediately obvious, and is more easily seen in the representation of $B$ as a $4$-holed sphere with gluings from \cite[Section 2]{FBRafi}, where $t$ is the front-to-back reflection. In any case, we can also conjugate $t$ by higher powers of $r$ to obtain these rotations of $L$. The fixed point set of $rtr^{-1}$ is then $r(N)$. Altogether, these loci cut $B$ into $32$ triangles with interior angles $\pi/2$, $\pi/4$, and $\pi/8$. By Courant's nodal domain theorem, we have $j \geq 31$.

The second $2$-dimensional irreducible representation of $\Isom(B)$ is given by
\begin{align*}
\rho(r) = \left(\begin{array}{cc}
    1 & 0 \\
    -1 & -1
\end{array}\right),&\quad
\rho(s) = \left(\begin{array}{cc}
    1 & 0 \\
    -1 & -1
\end{array}\right),\\
\rho(t) = \left(\begin{array}{cc}
    1 & 0 \\
    0 & 1
\end{array}\right),&\quad
\rho(u) = \left(\begin{array}{cc}
    -1 & -1 \\
    1 & 0
\end{array}\right)
\end{align*}
with respect to some basis for $V$. The function $f=\left(\begin{array}{c}   0  \\ 1 \end{array}\right)$ in this basis has the property that 
\[
\rho(r)^{2k}\rho(s) f = -f \quad \text{for} \quad k=0, \ldots, 3.
\]
Thus, $f$ vanishes on all the main diagonals of $O$. These divide $B$ into four congruent quadrilaterals. It follows that $f$ has at least four nodal domains, and hence that $j \geq 3$.
\end{proof}

Since the $1$- and $2$-dimensional irreducible real representations of $\Isom(B)$ cannot appear in the eigenspace for $\lambda_1(B)$, we conclude that $m_1(B) \geq 3$. Note that the proofs only relied on the symmetries of $B$. As such, the result holds even if we modify the metric on a fundamental $(2,3,8)$-triangle to get a different Riemannian metric on $B$ with the same isometry group.

\subsection{The upper bound}

To prove that $m_1(B) < 4$, we will need to use estimates on $\lambda_1(B)$. It is stated in \cite{Jenni} that  $\lambda_1(B) \in (3.83, 3.85)$ and a proof is sketched in \cite{JenniThesis}. These bounds were confirmed by Strohmaier--Uski \cite{StrohmaierUski}, who calculated that
\[
\lambda_1(B) \approx 3.8388872588421995185866224504354645970819150157\ldots
\]
up to a rigorous error bound of $10^{-6}$, though they believe that all the digits given above are correct. We will only need the much weaker enclosure $\lambda_1(B) \in [1,3.9].$

We will also require information about the length and number of the first few shortest closed geodesics in $B$. The systole of $B$ is equal to $2 \arccosh(1+\sqrt{2})$ and there are $12$ unoriented closed geodesics of this length \cite{Jenni,Schmutz}. An explicit formula for the $n$-th distinct entry $\ell_n$ in the length spectrum of $B$ was conjectured in \cite{AurichSteiner} and proved in \cite{ABS}. The multiplicities of these entries were estimated in \cite[Table 1]{AurichSteiner} and these values for $n\leq 67$ were confirmed by a rigorous algorithm in \cite{ABS}. Of these, we will only use the fact that $\ell_2 =2 \arccosh(3+2\sqrt{2})$ with unoriented multiplicity equal to $12$ and $\ell_3=2\arccosh(5+3\sqrt{2})$ with unoriented multiplicity equal to $24$ (the multiplicities given in \cite[Table 1]{AurichSteiner} count a curve and its inverse as distinct, so we divided their numbers by two). These lengths correspond to primitive closed geodesics since 
\[\ell_3 \approx 5.828071 < 6.114284 \approx 2 \ell_1. \]

In fact, we can describe the geodesics of length $\ell_1$, $\ell_2$, and $\ell_3$ geometrically, thereby confirming their multiplicity independently. By combining the $(2,3,8)$-triangles in the tiling of $B$ that share a common vertex with angle $\pi/8$, we obtain a tiling of $B$ by six regular octagons with interior angles $2\pi / 3$. Each closed geodesic of length $\ell_1$ passes through exactly two octagons in the tiling, connecting the midpoints of two opposite sides in each octagon. Each of the $6$ octagons has $4$ pairs of opposite sides and each such geodesic passes through $2$ octagons, so there are $6 \times 4 / 2 = 12$ such geodesics. Some examples are given by the geodesics connecting the midpoints of opposite sides of the fundamental octagon $O$ for the Bolza surface. The sides of $O$ are examples as well.

Each geodesic of length $\ell_2$ crosses two octagons along a main diagonal and follows the two common sides of another pair of octagons. There are again $6 \times 4 / 2 = 12$ such geodesics. Some examples are given by the main diagonals of $O$. From this, we see that there are isoceles triangles with a side of length $\ell_2/2$ opposite to a right angle and two sides of length $\ell_1/2$ opposite to angles $\pi/8$, allowing one to compute $\ell_1$ and $\ell_2$.

The geodesics of length $\ell_3$ are a bit more complicated to describe. One of them is illustrated in \figref{fig:bolza}. It intersects exactly three octagons, passing through the center of two of them but avoiding the center of the third, which it intersects in two distinct segments. The geodesic is preserved by the rotation of order $2$ around the center of the third octagon. Thus, there are exactly $4$ such geodesics per octagon (playing the role of the third octagon in the above description), for a total of $24$. To compute their length, observe that the center of $O$ forms an isosceles triangle with angle $3 \pi / 4$ opposite to a side of length $\ell_3/2$ and two sides of length $\ell_1/2$. We leave it to the interested reader to check that the resulting numbers $\ell_1$, $\ell_2$, and $\ell_3$ coincide with those given earlier.

\begin{figure}[htp]
\centering
\includegraphics{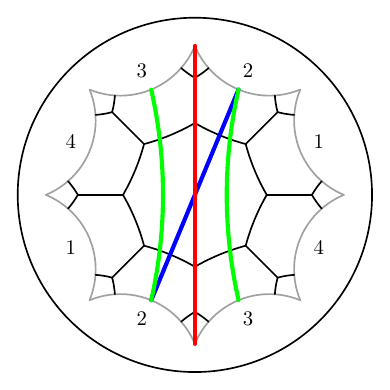}
\caption{A fundamental domain for the Bolza surface shown with its tessellation by regular octagons with interior angles $2\pi / 3$. The numbers indicate the side pairings. Some closed geodesics of length $\ell_1$, $\ell_2$, and $\ell_3$ are drawn in blue, red, and green respectively.}
\label{fig:bolza}
\end{figure}

Armed with precise knowledge of the bottom of the length spectrum of $B$, we can now prove a good upper bound for the number of eigenvalues in the interval $[1,3.9]$.
\begin{prop}
The number of eigenvalues of $\Delta_B$ in $[1,3.9]$ is strictly less than $4$.
\end{prop}
\begin{proof}
We apply \lemref{lem:upper} with $P$ equal to the set of closed geodesics of length $\ell_1$, $\ell_2$, and $\ell_3$, of which there are $12$, $12$, and $24$ respectively. We use a pair of polynomials $u$ and $v$ such that $u$ has a simple zero the the origin and a double zero at $50.46$, $v$ satisfies $v(1-1/4)=1 $ and has double zeros at $5.62$, $11.95$, $18.12$, and $31.92$, and such that $u(x^2)e^{-x^2/2}$ and $v(\xi^2)e^{-\xi^2/2}$ are Fourier transforms of each other. That these functions have the correct sign on the required intervals is checked in \texttt{verify8.sage}, which proves that $\Delta_B$ has at most $3.53617865617108 < 4$ eigenvalues in the interval.
\end{proof}

Together with the inequality $m_1(B) \geq 3$ shown in the previous subsections, this proves that $m_1(B)=3$ as claimed in \thmref{thm:Bolza}.

\bibliography{main}

\end{document}